\documentclass[a4paper,UKenglish,cleveref, autoref, thm-restate]{lipics-v2021}

\hideLIPIcs  %

\title{On Occurrence-Preserving Morphisms} %

\author{Kaisei Kishi}
{
Department of Information Science and Technology, 
Kyushu University,
Fukuoka, Japan 
}{kishi.kaisei.216@s.kyushu-u.ac.jp}{}{}

\newcommand{\isct}{%
M\&D Data Science Center, Institute of Integrated Research, 
Institute of Science Tokyo, Japan
}

\author{Peaker Guo}
{\isct}
{peakerguo@gmail.com}
{https://orcid.org/0000-0002-9098-1783}{}

\author{Cristian Urbina}
{Faculty of Mathematics, Informatics and Mechanics, University of Warsaw, Warsaw, Poland \and Center for Biotechnology and Bioengineering (CeBiB), Santiago, Chile}
{crurbina1997@gmail.com}
{https://orcid.org/0000-0001-8979-9055}{Polish National Science Center, grant no. 2022/46/E/ST6/00463; Basal Funds FB0001 and AFB240001, ANID, Chile; and FONDECYT Project 1-230755, ANID, Chile.}

\author{Hideo Bannai}{\isct}
{hdbn.dsc@tmd.ac.jp}
{https://orcid.org/0000-0002-6856-5185}{JSPS KAKENHI Grant Number JP24K02899}

\authorrunning{K. Kishi, P. Guo, C. Urbina, and H. Bannai} 

\Copyright{Kaisei Kishi, Peaker Guo, Cristian Urbina, and Hideo Bannai}



\category{} %

\relatedversion{} %

\nolinenumbers %

\EventEditors{Philip Bille and Nicola Prezza}
\EventNoEds{2}
\EventLongTitle{37th Annual Symposium on Combinatorial Pattern Matching (CPM 2026)}
\EventShortTitle{CPM 2026}
\EventAcronym{CPM}
\EventYear{2026}
\EventDate{June 15--17, 2026}
\EventLocation{Copenhagen, Denmark}
\EventLogo{}
\SeriesVolume{369}
\ArticleNo{22}

\usepackage{amsmath}
\usepackage{amssymb}
\usepackage{mathtools}
\usepackage[noadjust]{cite}
\usepackage{caption}

\usepackage[capitalise, noabbrev]{cleveref}
\usepackage[dvipsnames]{xcolor}

\usepackage{xspace}
\xspaceaddexceptions{]\}}

\renewcommand{\a}{\mathtt{a}}
\renewcommand{\b}{\mathtt{b}}
\renewcommand{\c}{\mathtt{c}}

\newcommand{\tm}{\ensuremath{\mathcal{T}}\xspace} 
\newcommand{\ltm}{\ensuremath{\tau}\xspace}
\newcommand{\flip}[1]{\ensuremath{\overline{#1}}\xspace}

\NewDocumentCommand{\occ}{m o}{%
  \IfNoValueTF{#2}
    {\ensuremath{\operatorname{\#occ}(#1)}}
    {\ensuremath{\operatorname{\#occ}_{#2}(#1)}}
}

\NewDocumentCommand{\Occ}{m o}{%
  \IfNoValueTF{#2}
    {\ensuremath{\operatorname{occ}(#1)}}
    {\ensuremath{\operatorname{occ}_{#2}(#1)}}
}

\newcommand{\mus}[1]{\ensuremath{\textsf{MUS}\left(#1\right)}\xspace}

\newcommand{\ext}[1]{\ensuremath{\textsf{EXT}\left(#1\right)}\xspace}

\newcommand{\no}[1]{\ensuremath{\textsf{NO}\left(#1\right)}\xspace}

\newcommand{\IF}{interference-free\xspace}
\newcommand{\SIF}{strongly interference-free\xspace}

\newcommand{\images}[1]{\ensuremath{\textsf{Im}\left(#1\right)}\xspace}

\newcommand{\pref}[1]{\ensuremath{#1^\mathit{pref}}\xspace}
\newcommand{\suf}[1]{\ensuremath{#1^\mathit{suf}}\xspace}

\usepackage[linesnumbered, ruled, vlined]{algorithm2e}

\SetKwInOut{Input}{Input}
\SetKwInOut{Output}{Output}
\SetKw{Not}{not}

\SetCommentSty{commentfont}
\DontPrintSemicolon

\newcommand{\str}[1]{\ensuremath{\textsl{str}(#1)}\xspace}
\newcommand{\ppref}{\ensuremath{P_{\textsl{pref}}}\xspace}
\newcommand{\psuf}{\ensuremath{P_{\textsl{suf}}}\xspace}

\newcommand{\duetopagelimit}[1]{#1}

\begin{document}

\maketitle

\begin{abstract}
A \emph{morphism} is a mapping that transforms words through letter-wise substitution, 
where each symbol is consistently replaced by a fixed word.
In the field of combinatorics on words,
one topic that has attracted considerable attention 
is the characterization of morphisms that preserve specific properties,
such as overlap-freeness, square-freeness, lexicographic order, and primitivity. 
Continuing this direction, we initiate the study on \emph{occurrence-preserving morphisms},
which address the following fundamental question: 
given a morphism $\phi$, two words $u$ and $v$, and $k \geq 1$,
under what conditions does the number of occurrences of $u$ in $v$ equal 
the number of occurrences of  $\phi^k(u)$ in  $\phi^k(v)$?
To answer this question, we introduce the notion of \emph{interference-free morphisms},
examine their properties, 
develop an efficient algorithm for deciding interference-freeness,
and uncover a connection to \emph{recognizable morphisms}.
We then present a precise characterization of occurrence-preserving morphisms in terms of interference-freeness.
As applications of our characterization, 
we first show that there exists a bijection between the starting positions of the occurrences of $u$ in $v$ and those of $\phi^k(u)$ in $\phi^k(v)$.
We then apply the characterization to the Fibonacci and Thue-Morse words
to identify their \emph{minimal unique substrings~(MUSs)}.
Finally, we exploit the connection between MUSs and \emph{net occurrences} to simplify existing proofs on net occurrences in these words.
\end{abstract}

\clearpage
\setcounter{page}{1}

\section{Introduction}

A morphism is a structure-preserving mapping, 
where the preserved structure is concatenation in the context of combinatorics on words.
More precisely, it is a mapping that transforms words through letter-wise substitution, with each symbol consistently replaced by a fixed image word.
Morphisms are fundamental objects that 
have been used to 
define infinite sequences and generate repetitive patterns \cite{book/Rozenberg1976,journal/tcs/2025/navarrourbina}.
Notable classes of morphisms include 
\emph{injective morphisms}, which play an important role in coding theory, and
\emph{recognizable morphisms},
which have been extensively studied in dynamical systems and formal language theory~\cite{journal/etds/2023/beal}.
More recently, the study of morphisms in relation to repetitiveness measures has gained considerable attention~\cite{conf/dlt/2022/frosini, conf/mfcs/2025/FRSU, conf/cpm/2023/fici, conf/iwoca/2019/brlek}.

A recurring theme in the literature on morphisms 
is identifying  properties that 
remain preserved under repeated applications of a morphism.
Such studies are particularly useful for generating infinite words that
guarantee desired combinatorial properties.
For example, Berstel and S{\'{e}}{\'{e}}bold~\cite{journal/dam/1993/berstel} showed that 
a morphism $h$ maps overlap-free words to overlap-free words
if and only if $h(\texttt{abbabaab})$ is overlap-free.
Subsequent work on overlap-free morphisms includes~\cite{journal/dam/1999/richomme, journals/dam/2004/richomme}.
Other morphism-preserved properties have also been studied, such as
lexicographic order~\cite{journal/dmtcs/2007/richomme},
palindromic richness~\cite{journal/fuin/2022/dolce},
Abelian power-freeness~\cite{journal/ijac/1993/carpi, journal/tcs/2009/keranen},
square-freeness~\cite{hsiao2003square},
primitivity~\cite{journal/afp/2023/holub,journal/dm/2008/huang},
BWT runs~\cite{conf/mfcs/2025/FRSU}, and
Sturmian words~\cite{conf/mfcs/1993/berstel}.
Continuing this line of research, we initiate the study of \emph{occurrence-preserving morphisms},
which address the following fundamental question: 
given a morphism $\phi$, two words $u$ and $v$, and $k \geq 1$,
\begin{center}
\begin{minipage}{0.7\textwidth}
\centering
\emph{
under what conditions does the number of occurrences of $u$ in $v$ equal 
the number of occurrences of $\phi^k(u)$ in $\phi^k(v)$?
}
\end{minipage}
\end{center}

Understanding this question has direct applications in 
proving properties that are constrained by the number of occurrences of factors (substrings),
such as \emph{minimal unique substrings (MUSs)} and \emph{net occurrences}.
A MUS is a unique substring whose every proper substring is repeated,
whereas  a net occurrence is an occurrence of a repeated substring whose every proper super-string is unique.
Both notions have been extensively studied for their combinatorial properties and efficient algorithms: see, for example,
\cite{journal/fuin/2011/ilie, journal/algorithmica/2022/mieno, conf/cpm/2017/mieno, conf/mfcs/2016/mieno} for MUSs,
and \cite{conf/cpm/2024/guo, conf/spire/2024/ohlebusch, conf/spire/2024/guo, DBLP:journals/corr/abs-2410-06837, conf/cpm/2025/mieno, conf/cpm/2026/KI} for net occurrences.

\subparagraph*{Our results.}
In this work, we make three main contributions.
First, to answer the above question, we introduce the notion of \emph{interference-free morphisms},
analyze their properties, 
develop an efficient algorithm for deciding interference-freeness,
and uncover a connection to recognizable morphisms  (\cref{thm:if-implies-rec}).
Second, we  provide a precise characterization of occurrence-preserving morphisms in terms of interference-freeness (\cref{thm:occ-preserve}).
Third, as an application, we apply this characterization to the Fibonacci and Thue-Morse words
to identify their \emph{minimal unique substrings~(MUSs)} (\cref{thm:mus-fib} and \cref{thm:mus-tm}).
We further exploit the connection between MUSs and \emph{net occurrences}~\cite{conf/cpm/2025/mieno} to simplify existing proofs on net occurrences in these words~\cite{conf/cpm/2025/guo}.
In the process, we establish new properties of these
morphisms and  words that may have independent interest.

\section{Preliminaries}\label{sec:preliminaries}

\subparagraph*{Basics.}
Let $\Sigma$ be an ordered alphabet.
We assume $\Sigma = \{ \a, \b\}$ when $|\Sigma|=2$.
A \emph{word (or string)} is an element of $\Sigma^*$. 
The length of a word $w$ is denoted as $|w|$.
Let $\varepsilon$ denote the \emph{empty word} of length 0.
We use $w[i]$ to denote the $i^{\text{th}}$ character of a word $w$.
Let $u \cdot v = u \ v$ denote the \emph{concatenation} of two words,~$u$ and~$v$.
A \emph{factor} (or \emph{substring}) of a word~$w$ of length $n$, 
starting at position~$i$ 
and ending at position~$j$, is written as $w[i \ldots j]$. 
A factor $w[1\ldots j]$ is called a \emph{prefix} of~$w$, 
while $w[i \ldots n]$ is called a \emph{suffix} of $w$.
A factor $u$ of $w$ is a \emph{proper} factor if $u \neq w$.
For two words $u$ and $w$,
let $\Occ{u}[w] = \{ i \; | \; w[i \ldots i+|u|-1] = u \}$ be the set of (starting positions of) \emph{occurrences} of $u$ in $w$,
and let $\occ{u}[w] = |\Occ{u}[w]|$.
When convenient, we identify an occurrence $i \in \Occ{u}[w]$ with its corresponding interval $[i, i+|u|-1]$ in $w$.
For convenience, we treat the empty word $\varepsilon$ as occurring $|w| + 1$ times in $w$: 
before position $1$, between positions $i$ and $i + 1$ for $1 \leq i \leq |w|$, and after position $|w|$. 
We represent the $i^\text{th}$ such occurrence by the interval $[i, i-1]$ for $1 \leq i \leq |w|+1$.
For a word $w$, let $w^R = w[|w|]\cdots w[1] $ denote the reverse of $w$.
For a non-empty word $w$, 
a sequence of non-empty words $(  x_k )^m_{k=1}$ %
is referred to as a \emph{factorization} of $w$ if
$w = x_1 \cdots x_m $.
For a word $w$, a word of the form $w[i \ldots |w|] \cdot w[1 \ldots i-1]$, for some $1 \leq i \leq |w|$,
is called a \emph{rotation} of $w$.  
Let $\mathcal{R}(w)$ denote the multiset of all $|w|$ rotations of $w$.
For a word $w$ with $|w| \geq 2$, let $w^\lhd = w[1 \ldots |w|-1]$ denote its longest proper prefix.

\subparagraph*{Morphisms.}
Let $\Sigma$ and $\Gamma$ be two alphabets.
A \emph{morphism} is a map $\phi : \Sigma^* \to \Gamma^*$ such that
$\phi(uv) = \phi(u) \ \phi(v)$ for all $u,v \in \Sigma^*$.
For each $c \in \Sigma$, $\phi(c)$ is called an \emph{image}\footnote{
In this paper, we use the term \emph{image} only for $\phi(c)$ with $c\in \Sigma$,
and not for $\phi(u)$ with $u \in \Sigma^*$.
} of $\phi$,
and let $\images{\phi} = \{ \phi(c) \; | \; c \in \Sigma \}$ be the set of images of $\phi$.
Letting $\Sigma = \{\alpha_1, \ldots, \alpha_{\sigma}\}$,
$\phi$ can be equivalently specified by 
$\alpha_1 \mapsto \phi(\alpha_1), \ldots, \alpha_{\sigma} \mapsto \phi(\alpha_{\sigma})$.
Let $\phi^0(u) = u$ and let $\phi^i(u) = \phi(\phi^{i-1}(u))$ for each $i \geq 1$.
Further,
$\phi$ is \emph{non-erasing} if $\phi(a) \neq \varepsilon$ for all $a \in \Sigma$;
$\phi$ is \emph{injective}\footnote{
Note that this defines a stronger notion of injectivity than simply requiring that
for all $c \neq c' \in \Sigma$, $\phi(c) \neq \phi(c')$.
} 
if $\phi(u) \neq \phi(v)$ for all $u \neq v \in \Sigma^*$;
$\phi$ is \emph{$\ell$-uniform} if $|\phi(a)| = \ell$ for all $a \in \Sigma$.
For a binary word $w$,
let $\flip{w}$ denote the word obtained by
applying the morphism $\a \mapsto \b, \b \mapsto \a$.

\subparagraph*{Fibonacci and Thue-Morse morphisms and words.}
Let $\varphi$ denote the \emph{Fibonacci morphism} defined by
$\varphi(\texttt{a}) = \texttt{ab}$ and $\varphi(\texttt{b}) = \texttt{a}$.
Let $\mu$ denote the \emph{Thue-Morse morphism} defined by
$\mu(\texttt{a}) = \texttt{ab}$ and $\mu(\texttt{b}) = \texttt{ba}$.
For each $i \geq 1$,
let $F_i = \varphi^{i-1}(\texttt{b})$ be the (finite) \emph{Fibonacci word of order $i$}.
For each $i \geq 1$,
let $\tm_i = \mu^{i-1}(\texttt{a})$ be the (finite) \emph{Thue-Morse word of order $i$}.
The Fibonacci and Thue-Morse words can also be obtained as follows:
$F_1 = \texttt{b}, F_2 = \texttt{a}$, and 
$F_i = F_{i-1} \ F_{i-2}$ for each $i \geq 3$;
$\tm_1 = \texttt{a}$ and
$\tm_i = \tm_{i-1} \ \flip{ \tm_{i-1}}$ for each $i \geq 2$.
Further, for each $i \geq 1$, let $f_i = |F_i|$, which equals the $i^{\text{th}}$ Fibonacci number,
and let $\ltm_i = |\tm_i| = 2^{i-1}$.
We next review the following known properties of Fibonacci words.

\begin{observation}[\cite{journal/tit/2021/navarro, conf/cpm/2025/guo}]\label{obs:fib-properties}
The following hold for $F_i$.
\begin{itemize}
\item Neither \texttt{aaa} nor \texttt{bb} occurs in any $F_i$.
\item 
Let $\Delta_0 = \texttt{ba}$, $\Delta_1 = \texttt{ab}$, 
and $\Delta_i = \Delta_{(i \bmod 2)}$ for $i \geq 2$. 
Define $G_i = F_i[1 \ldots f_i-2]$ for each $i \geq 3$.
Then, $F_{i} = G_{i} \ \Delta_i$
and $\Occ{G_{i-1}}[F_i] = \{ 1, f_{i-2}+1 \}$ for each $i \geq 7$.
\end{itemize}
\end{observation}

\subparagraph*{MUSs and net occurrences.}
Consider a string $w$ and a unique substring $u$ of $w$.
Let $[i, j]$ be the only occurrence of $u$ in $w$.
We say $u$ is a \emph{minimal unique substring (MUS)} of $w$
if both strings $w[i+1 \ldots j]$ and $w[i \ldots j-1]$ are repeated in $w$.
Let $\mus{w}$ denote the set of MUSs of $w$.
An occurrence $[i, j]$ in $w$ is a \emph{net occurrence} if
the corresponding string 
$w[i   \ldots j  ]$ is repeated,
while both left extension 
$w[i-1 \ldots j  ]$ and right extension
$w[i   \ldots j+1]$ are unique.
When $i=1$, $w[i-1 \ldots j  ] $ is assumed to be unique;
when $j=|w|$, $w[i   \ldots j+1] $ is assumed to be unique.
Let $\no{w}$ denote the set of net occurrences in $w$.

\section{Interference-Free Morphisms}\label{sec:if-morphisms}

To fully characterize occurrence-preserving morphisms,
we first introduce the notion of \emph{\IF} morphisms and establish some of their key properties in this section.
To motivate the definition, we begin with an example illustrating that,
after applying a morphism $\phi$ to words $u$ and $v$,
the occurrences of $u$ in $v$ may fail to be preserved
when certain forms of ``interference'' take place between $\phi(u)$ and images of $\phi$.

\begin{example}\label{example:fib-tm-morphisms}
In (1) and (2) of \cref{fig:abaab-example},
we have $\occ{u}[v] = \occ{\phi(u)}[\phi(v)] = 2$,
and $\occ{u}[v] = 2 < 3 = \occ{\phi(u)}[\phi(v)]$.
A third occurrence of $\phi(u)$ (underlined in red in (2)) emerges in $\phi(v)$
because $\phi(u) = \a\b \cdot \a$ and $\a$ is a proper prefix of $\phi(\a) = \a\b$.
In (3) and (4) of  \cref{fig:abaab-example},
we have $\occ{u}[v] = 1 < 2  = \occ{\phi(u)}[\phi(v)]$,
and $\occ{u}[v] =  \occ{\phi(u)}[\phi(v)] = 1$.
A second occurrence of $\phi(u)$ (underlined in red in (3)) emerges in $\phi(v)$ 
because $\phi(u) = \a \cdot \b \a \cdot \b$, $\a$ is a proper suffix of $\phi(\b) = \b \a$, and $\b$ is a proper prefix of $\phi(\b) = \b\a$. 
\end{example}

\begin{figure}[t]
\centering
\includegraphics[width=0.9\linewidth]{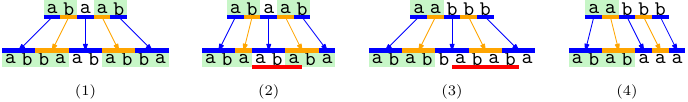}
\caption{
Illustration of \cref{example:fib-tm-morphisms}.
In each of (1)--(4), we show $v$ at the top, $\phi(v)$ at the bottom, 
and in the middle, arrows indicate the mapping from each $v[i]$ to $\phi(v[i])$;
the substring $u$ in $v$, and the corresponding $\phi(u)$ in $\phi(v)$, are highlighted in green.
Specifically, 
in (1) and (2), $u = \a \b$ and $v = \texttt{abaab}$, whereas
in (3) and (4), $u = \a \a$ and $v = \texttt{aabbb}$.
Moreover, $\phi: \a \mapsto \a\b, \b \mapsto \b\a$  %
in (1) and (3);
$\phi: \a \mapsto \a\b, \b \mapsto \a$ %
in (2) and (4).
}
\label{fig:abaab-example}
\end{figure}

To start formalizing the idea from the motivating example, we first define the following
and review a key property of injective morphisms. 

\begin{definition}[Image Factorizations]
Let $\phi: \Sigma^* \to \Gamma^*$ be a morphism and let $w \in \Gamma^*$ be a word.
We say $w$ admits a \emph{$\phi$-image factorization} if
$w = X_1 \cdots X_n$, where each $X_i \in \images{\phi}$.
When the morphism is clear from context, we simply say that $w$ admits an \emph{image factorization}.
\end{definition}

\begin{lemma}[\cite{book/daglib/0025093}]\label{lem:inj-uniq-fac}
Let $\phi : \Sigma^* \to \Gamma^*$ be a  morphism and let $u \in \Sigma^*$ be a word.
If $\phi$ is injective, then $\phi(u)$ admits a unique image factorization.
\end{lemma}

\duetopagelimit{
A set $X$ of words is a \emph{code} if each non-empty word $w \in X^*$ admits a unique factorization into words of $X$.
The above lemma implies that if $\phi$ is injective, then $\images{\phi}$ forms a code.
}

We now give precise definitions of what was previously referred to as ``interference''.

\begin{definition}[Interfered Image Factorizations]
Let $\phi: \Sigma^* \to \Gamma^*$ be a morphism and let $w \in \Gamma^*$ be a word.
We say $w$ admits an \emph{interfered $\phi$-image factorization} 
if $w = x \cdot y \cdot z$, where
$x$ is a \emph{proper} suffix of some image in $\images{\phi}$,  
$y$ admits an image factorization, 
$z$ is a \emph{proper} prefix of some image in $\images{\phi}$, and
$x \cdot z \neq \varepsilon$.
When the morphism is clear from context, we simply say that $w$ admits an \emph{interfered image factorization}.
\end{definition}

In \cref{example:fib-tm-morphisms}, 
$\a\b \cdot \a$ is an interfered image factorization of $\varphi(\a\b)$ and 
$\a \cdot \b\a \cdot \b$ is an interfered image factorization of $\mu(\a\a)$.
Beyond the cases illustrated in \cref{example:fib-tm-morphisms} and formalized in \cref{def:if-morph}, another type of ``interference'' can occur. 
We give an example and then formalize this notion below.

\begin{example}\label{example:var-tm}
Let $\phi$ be a variant of the Thue-Morse morphism~\cite[A036577]{oeis} defined as 
 $\phi(\a) = \a\b\c$, $\phi(\b) = \a\c$, $\phi(\c) = \b$;
let $u = \c$, and $v = \c\a$.
Then, $\phi(u) = \b$, $\phi(v) = \b \a \b \c$, 
and $\occ{u}[v] = 1 <  2 = \occ{\phi(u)}[\phi(v)]$.
\end{example}

\begin{definition}[Inner Image Factor]
Let $\phi: \Sigma^* \to \Gamma^*$ be a morphism and let $w \in \Gamma^*$ be a word.
We say $w$ is an \emph{inner $\phi$-image factor}
if $w$ is a proper factor of an image $\phi(c)$ for some $c \in \Sigma$, but is neither a prefix nor a suffix of $\phi(c)$. 
When the morphism is clear from context, we simply say that $w$ is an \emph{inner image factor}.
\end{definition}

\noindent
In \cref{example:var-tm}, $\b$ is a proper factor of $\phi(\a) = \a\b\c$,
but is neither a prefix nor a suffix of $\a\b\c$.
Hence $\b$ is an inner image factor.
Having formalized the two types of ``interference'', we are now ready to define \IF morphisms.

\begin{definition}[Interference-Free Morphisms]\label{def:if-morph}
Let $\phi: \Sigma^* \to \Gamma^*$ be an injective morphism and 
let $\mathcal{L} \subseteq \Sigma^*$.
We say $\phi$ is \emph{interference-free on $\mathcal{L}$} 
if for every non-empty word $u \in \mathcal{L}$, 
\begin{itemize}
\item 
$\phi(u)$ does \emph{not} admit an interfered image factorization, and
\item 
$\phi(u)$ is \emph{not} an inner image factor.
\end{itemize}
If $\mathcal{L} = \Sigma^*$, we say that $\phi$ is \emph{strongly} interference-free.
\end{definition}

The definition is illustrated on the left of \cref{fig:inj-vs-if}.
(The case where $\phi(u)$ is an inner image factor can be regarded as an edge case 
that does not affect the intuition, and thus omitted from the figure.)
We now observe the following on the Fibonacci morphism.

\begin{observation}\label{obs:fib-not-if-odd}
$\varphi$ is \emph{not} \IF on $\{ F_i : i \geq 5 \text{~and~} i \text{~is odd} \}$
\end{observation}

\begin{proof}
Let $n = |F_i|$ and 
consider the factorization of $\varphi(F_i)= F_{i+1} = X_1 \cdots X_n$
where each $X_j = \varphi(F_i[j]) \in \images{\varphi}$ for $1 \leq j \leq n$.
Note that $i+1$ is even. 
By \cref{obs:fib-properties}, $\a \b \a$ is a suffix of even Fibonacci words.
Thus, $X_{n-1} = \varphi(\a) = \a\b$ and $X_n = \varphi(\b) = \a$.
It follows that $\varphi(F_i)$ admits an interfered image factorization  
$\varphi(F_i)= x \cdot y \cdot z$ with $x = \varepsilon$,
$y = X_1 \ldots X_{n-1}$, and $z = X_n = \a$ being a proper prefix of $\varphi(\a) = \a\b$.
Therefore, $\varphi$ is not \IF on the desired set.
\end{proof}

\noindent
This, combined with
$\mu$  not being \IF on $\{ \a \a \}$ (\cref{example:fib-tm-morphisms}),
gives the following.

\begin{observation}\label{obs:fib-tm-not-sif}
The Fibonacci morphism $\varphi$ and Thue-Morse morphism $\mu$ are not \SIF.
\end{observation}

\subparagraph*{An algorithm for deciding interference-freeness.}

We next present an efficient algorithm for deciding interference-freeness.
At a high level, the algorithm employs a dynamic programming approach 
supported by an \emph{Aho-Corasick automaton}~\cite{journal/cacm/1975/aho}.
The proof  is included in Appendix~\ref{appendix:algo}.

\begin{theorem}\label{thm:if-algo}
Let $\phi : \Sigma^* \to \Gamma^*$ be an injective morphism
and let $u \in \Sigma^*$.
Define $m = \sum_{x \in \images{\phi}} |x|$,
$w = \phi(u)$, $n = |w|$, and
 $\mathit{occ} = \sum_{x \in \images{\phi}} \occ{x}[w]$.
Then, whether $\phi$ is \IF on $\{u\}$ 
can be decided in $O(m + n  + \mathit{occ})$ expected time.
\end{theorem}

\subsection{Properties of Interference-Free Morphisms}

\begin{figure}[t]
\centering
\includegraphics[width=0.75\linewidth]{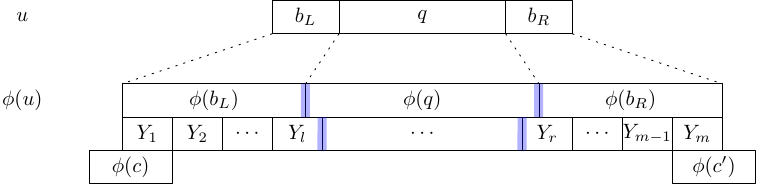}
\caption{Illustration of Case (1) in the proof of \cref{lem:if-barrier}.}
\label{fig:if-barrier}
\end{figure}

In this subsection, we examine several key properties of \IF morphisms.
When proving interference-freeness, 
it is often more convenient to use the following lemma rather than verify the definition directly.
In what follows, $b_L$ and $b_R$ serve as the left and right interference ``barriers''.

\begin{lemma}\label{lem:if-barrier}
Let $\phi : \Sigma^* \rightarrow \Gamma^*$ be an  injective morphism and let $u \in \Sigma^*$.
If there exist $b_L, q, b_R \in \Sigma^*$ such that
$u = b_L \cdot q \cdot b_R$ and $\phi$ is \IF on $\{b_L, b_R\}$,
then $\phi$ is \IF on $\{u\}$.
\end{lemma}

\begin{proof}
We proceed by contrapositive and assume that $\phi$ is not \IF on $\{u\}$.
Then one of the following two cases occurs.
\begin{enumerate}[(1)] 
\item\label{case:if-barrier-factorization}
$\phi(u) = \phi(b_L) \cdot \phi(q) \cdot \phi(b_R)$ 
admits an interfered image factorization $\phi(u) = Y_1 \cdots Y_m$,
where $Y_1$ is a proper suffix of $\phi(c)$ for some $c \in \Sigma$,
$Y_i \in \images{\phi}$ for each $2 \leq i \leq m-1$,
$Y_m$ is a proper prefix of $\phi(c')$ for some $c' \in \Sigma$,
and $Y_1 \cdot Y_m \neq \varepsilon$;
or 
\item\label{case:if-barrier-factor}
$\phi(u)$ is an inner image factor.
\end{enumerate}
In Case~(\ref{case:if-barrier-factorization}),
if $\phi(u)$ admits an interfered image factorization, see \cref{fig:if-barrier} for an illustration of this case.
Let $l$ be the smallest index such that $\phi(b_L)$ is a prefix of $Y_1 \cdots Y_l$, and
let $r$ be the largest index such that $\phi(b_R)$ is a suffix of $Y_r \cdots Y_m$.
(Note that these do not have to be proper prefix/suffix relations; 
the non-proper case corresponds to the situation where the blue vertical lines are aligned in \cref{fig:if-barrier}.)
Since at least one of $Y_1$ and $Y_m$ is non-empty,
when $Y_1 \neq \varepsilon$, $\phi$ is not \IF on $\{b_L\}$;
when $Y_m \neq \varepsilon$, $\phi$ is not \IF on $\{b_R\}$.
In Case~(\ref{case:if-barrier-factor}),
if $\phi(u)$ is an inner image factor, then $\phi(b_L)$ and $\phi(b_R)$ are also inner image factors, since both are factors of $\phi(u)$. 
Therefore, we have shown by contrapositive that $\phi$ is \IF on $\{u\}$.
\end{proof}

With \cref{lem:if-barrier}, one can often verify that $\phi$ is \IF on $\{u\}$ simply
by examining suitable choices of $b_L$ and $b_R$.
We next demonstrate the usefulness of \cref{lem:if-barrier} through the following corollaries.

\begin{corollary}\label{cor:fib-if-even}
$\varphi$ is \IF on $\mathcal{L}_{F} = \{ F_i : i \geq 4 \text{~and~} i \text{~is even}  \}$.
\end{corollary}

\begin{proof}
First note that, for any $i \geq 2$, $\a\b$ is a prefix of $\varphi(F_i)$.
Next, for each $F_i \in \mathcal{L}_F$,
$\varphi(F_i) = F_{i+1}$ is an odd Fibonacci word.
Thus, by \cref{obs:fib-properties},
$\a\b$ is also a suffix of $\varphi(F_i)$.
Since $\varphi$ is \IF on $\{ \a\b \}$ by definition,
it follows from \cref{lem:if-barrier} that $\varphi$ is also \IF  on $\mathcal{L}_F$.
\end{proof}

\begin{corollary}\label{cor:tm-if}
$\mu$ is \IF on $\mathcal{L}_{\tm} = \{ \tm_i : i \geq 4 \}$.
\end{corollary}

\begin{proof}
For each $\tm_i \in \mathcal{L}_{\tm}$,
$\mu(\tm_i) = \tm_{i+1}$.
Observe that \texttt{abba} is a prefix of $\mu(\tm_i)$,
\texttt{baab} is a suffix of $\mu(\tm_i)$ if $i$ is odd,
and \texttt{abba} is a suffix of $\mu(\tm_i)$ if $i$ is even.
Since $\mu$ is \IF on $\{ \texttt{abba}, \texttt{baab} \}$ by definition,
it follows from  \cref{lem:if-barrier} that $\mu$ is \IF on $\mathcal{L}_{\tm}$.
\end{proof}

\cref{lem:if-barrier} also leads to the following 
 simple yet powerful characterization.

\begin{lemma}\label{lem:if-alphabet-sif}
An injective morphism 
$\phi : \Sigma^* \rightarrow \Gamma^*$
is \SIF if and only if $\phi$ is \IF on $\Sigma$.
\end{lemma}

\begin{proof}
$(\Rightarrow)$
If $\phi$ is \SIF, then $\phi$ is \IF on $\Sigma^*$,
in particular, $\phi$ is \IF on $\Sigma$.
$(\Leftarrow)$
If $\phi$ is \IF on $\Sigma$,
then,  for each $u \in \Sigma^*$,
we know that $\phi$ is \IF on $\{ u[1], u[|u|] \}$.
It follows that, by \cref{lem:if-barrier}, $\phi$ is \IF on $\{u\}$.
Thus, $\phi$ is \SIF.
\end{proof}

With~\cref{lem:if-alphabet-sif}, we can alternatively prove~\cref{obs:fib-tm-not-sif}
by showing that the Fibonacci morphism $\varphi$ is not \IF on $\{ \b \}$ and 
the Thue-Morse morphism $\mu$ is not \IF on $\{ \a \}$.
Further, the following result also follows directly from~\cref{lem:if-alphabet-sif}.

\begin{corollary}
The following injective  morphisms are \SIF: %
\begin{itemize}
\item 
Mephisto-Waltz~\cite[A064990]{oeis}: $\a \mapsto \a\a\b$, $\b \mapsto \b\b\a$;
\item 
Thue-Morse-Morse~\cite[A189718]{oeis}: $\a \mapsto \a\b\b$, $\b \mapsto \b\a\a$;
\item 
Last nonzero digit~\cite[A080846]{oeis}: $\a \mapsto \a\b\a$, $\b \mapsto \a\b\b$.
\end{itemize}
\end{corollary}

\subsection{Interference-Free and Recognizable Morphisms}

In this subsection, we uncover a connection between \IF and  recognizable morphisms.
We adapt the definition from \cite{conf/mfcs/2025/FRSU}, while
extracting the core idea as follows.

\begin{definition}[Circular Image Factorization]
Let $\phi: \Sigma^* \to \Gamma^*$ be a morphism and let $w \in \Gamma^*$ be a word.
We say $w$ admits a \emph{circular $\phi$-image factorization}
if $w = q \cdot r \cdot p$, where 
$r$ admits an image factorization,
$p \cdot q = \phi(c)$ for some $c \in \Sigma$, \emph{and}\footnote{
The condition $p \neq \varepsilon$ is necessary for uniqueness of circular image factorizations.
Otherwise, any image factorization $w = X_1 \cdots X_n$ 
would always yield at least two distinct circular image factorizations,
 by setting $q = \varepsilon$, $r = X_1 \cdots X_{n-1}$, and $p = X_n$;
or $q = X_1$, $r = X_2 \cdots X_{n}$, and $p = \varepsilon$. 
} $p \neq \varepsilon$. 
When the morphism is clear from the context, we simply say that $w$ admits a \emph{circular image factorization}.
\end{definition}

\begin{definition}[Recognizable Morphisms]\label{def:recognizable}
Let $\phi: \Sigma^* \to \Gamma^*$ be an injective morphism  
and let $\mathcal{L} \subseteq \Sigma^*$.
We say $\phi$ is \emph{recognizable on $\mathcal{L}$} 
if for every non-empty word $u \in \mathcal{L}$ and every rotation $w' \in \mathcal{R}(w)$, where $w = \phi(u)$,
$w'$ admits a unique circular image factorization.
If $\mathcal{L} = \Sigma^*$, we simply say that $\phi$ is recognizable.
\end{definition}

\begin{figure}[t]
\centering
\includegraphics[width=\linewidth]{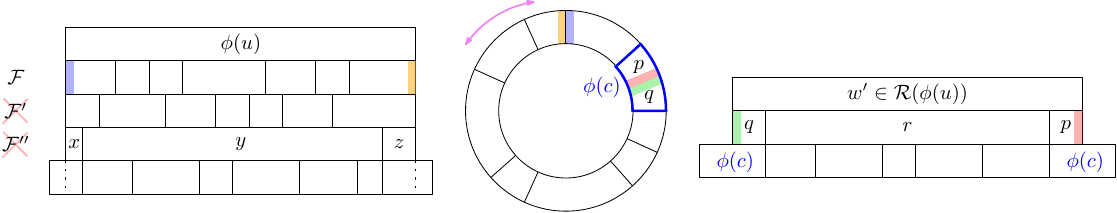}
\caption{
Various illustrations 
of concepts and results 
from \cref{sec:if-morphisms}.
Each unlabeled solid rectangle or annular sector represents an image of $\phi$.
\underline{Left:} for 
\cref{lem:inj-uniq-fac} and  \cref{def:if-morph}.
If $\phi$ is injective, then $\phi(u)$ admits the unique image factorization $\mathcal{F} = \phi(u[1]) \cdots \phi(u[|u|])$, and no alternative image factorization $\mathcal{F}'$ is possible.
If, moreover, $\phi$ is \IF on $\{u\}$, then no interfered image factorization $\mathcal{F}''$ is possible either.
\underline{Right:} for \cref{def:recognizable}.
Given a rotation $w'$ and a circular image factorization of $w'$,
by folding the (linear) word and connecting its two ends (marked in green and red), we obtain the circular word shown in the middle.
\underline{Middle:} for \cref{remark:if-rec-intuition}.
}
\label{fig:inj-vs-if}
\end{figure}

\begin{remark}\label{remark:if-rec-intuition}
The  intuition behind \cref{thm:if-implies-rec} is illustrated 
in \cref{fig:inj-vs-if}.
If $\phi$ is \IF on $\{u\}$,
then by folding the (linear) word $\phi(u)$ and connecting its two ends (marked in blue and yellow on the left),
we obtain the circular word shown in the middle.
Crucially, regardless of how the circular word is rotated, producing different rotations in $\mathcal{R}(\phi(u))$,
the condition for recognizability remains satisfied.
Hence, with interference-freeness, each unique (linear) image factorization
induces a unique circular image factorization.
\lipicsEnd
\end{remark}

\begin{theorem}\label{thm:if-implies-rec}
Let $\phi: \Sigma^* \to \Gamma^*$ be an injective morphism and let $\mathcal{L} \subseteq \Sigma^*$. 
If $\phi$ is \IF on $\mathcal{L}$, then $\phi$ is recognizable on $\mathcal{L}$.  
\end{theorem}

\begin{proof}
We proceed by contrapositive and prove that,
for each non-empty word $u \in \mathcal{L}$,
if $\phi$ is \emph{not} recognizable on $\{u\}$,
then $\phi$ is \emph{not} \IF on $\{u\}$.

Let $w = \phi(u)$.
We have assumed that $\phi$ is not recognizable on $\{u\}$,
which means there exist some $w' \in \mathcal{R}(w)$ such that 
either $w'$ does not admit a circular image factorization, 
or its factorization is not unique.
We first show that each $w' \in \mathcal{R}(w)$ does admit a circular image factorization.
Let $w = X_1 \cdots X_n$ be an image factorization.
Since $w'$ is a rotation of $w$, there exists an index $1 \leq c \leq n$ 
such that  $w' = \suf{X_c} \cdot X_{c+1} \cdots X_n \cdot X_1 \cdots X_{c-1} \cdot \pref{X_c}$, where
$\pref{X_c} \neq \varepsilon$ and  $\pref{X_c} \cdot \suf{X_c} = X_c$.
Setting $q = \suf{X_c}$, $r = X_{c+1} \cdots X_n \cdot X_1 \cdots X_{c-1}$, and $p = \pref{X_c}$,
we have shown that $w'$ admits a circular image factorization. 

Since such a factorization exists for each $w' \in \mathcal{R}(w)$, 
our assumption implies that there must exist $w'$ whose  circular image factorization is not unique.
Hence, consider  one such factorization of $w'$, given by
$w' = \suf{Y_1} \cdot Y_2 \cdots Y_m \cdot \pref{Y_1}$,
where $Y_i \in \images{\phi}$ for $1 \leq i \leq m$ and $\pref{Y_1} \cdot \suf{Y_1} = Y_1$. 
Further, define index $1 \leq k \leq m$, words $\pref{Y_{k}}$ and \suf{Y_{k}}
such that $\pref{Y_{k}} \cdot \suf{Y_{k}} = Y_k$, and
$ \suf{Y_1} \cdot Y_2 \cdots \pref{Y_{k}}  =  \suf{X_c} \cdot X_{c+1} \cdots X_n $.
Now, rotating $w'$ back to $w$ gives
$w = \suf{Y_{k}} \cdot Y_{k+1} \cdots Y_m \cdot Y_1 \cdots Y_{k-1} \cdot \pref{Y_{k}}$.
We now prove that $\phi$ is not \IF on $\{u\}$ by considering this factorization of $w$ in three cases.
\begin{enumerate}[(1)] 
\item 
$\pref{Y_{k}} = Y_k$ and $\suf{Y_k} = \varepsilon$. 
In this case, $w = X_1 \cdots X_n =  Y_{k+1} \cdots Y_m \cdot Y_1 \cdots Y_{k-1} \cdot Y_{k}$.
The injectivity of $\phi$ implies that the factorizations for $w$ must be equivalent, contradicting that the two factorizations for $w'$ were distinct.
\item 
$\pref{Y_{k}} = \varepsilon$ and $\suf{Y_k} = Y_k$. 
By a symmetric argument, $\phi$ is again not injective.
\item 
$\pref{Y_{k}} \neq \varepsilon$ and $\suf{Y_k} \neq \varepsilon$. 
In this case, let $x = \suf{Y_{k}}$, $y =  Y_{k+1} \cdots Y_m \cdot Y_1 \cdots Y_{k-1}$, and $z = \pref{Y_{k}}$. 
Then, $\phi$ is not \IF on $\{u\}$ since $x$ is a proper suffix of $Y_k \in \images{\phi}$, $y \in \images{\phi}^*$, $z$ is a proper prefix of $Y_k \in \images{\phi}$, and $x \cdot z = Y_k \neq \varepsilon$. 
Thus, $w = \phi(u)$ admits an interfered image factorization.
\end{enumerate}
Therefore, we have shown that if $\phi$ is not recognizable on $\{u\}$,
then $\phi$ is not \IF on $\{u\}$.
This completes the proof.
\end{proof}

\begin{figure}[t]
\centering
\includegraphics[width=0.9\linewidth]{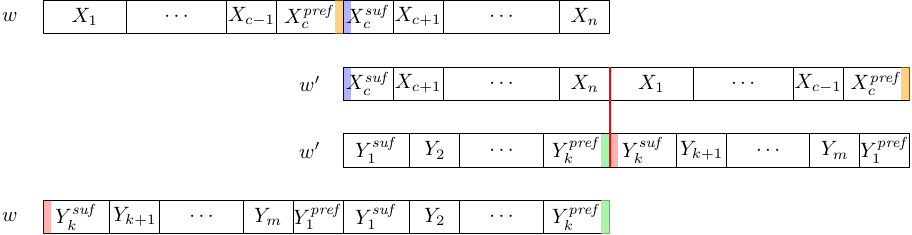}
\caption{Illustration of the proof of \cref{thm:if-implies-rec}.
}
\label{fig:if-rec-proof}
\end{figure}

\cref{thm:if-implies-rec} naturally extends to the following, more general result.

\begin{corollary}
Every \SIF  morphism is recognizable.
\end{corollary}

Having established that interference-freeness implies recognizability, 
it is natural to ask whether the converse also holds.
In \cite{journal/etds/2023/beal}, the Fibonacci morphism was shown to be recognizable. 
Together with \cref{obs:fib-not-if-odd}, this implies 
a morphism $\varphi$ and an infinite family $\mathcal{L}$ of words 
such that $\varphi$ is recognizable on $\mathcal{L}$ but not \IF on $\mathcal{L}$.
Thus, recognizability does not imply interference-freeness.
We summarize the resulting strict hierarchy below.

\begin{remark}\label{remark:hierarchy}
The following inclusion relation among morphisms holds:
\[
\mathtt{INTERFERENCE}\_\mathtt{FREE} 
\subsetneq \mathtt{RECOGNIZABLE} 
\subsetneq  \mathtt{INJECTIVE}.
\]
Moreover, the Thue-Morse morphism is injective but not recognizable, 
whereas the Fibonacci morphism is recognizable but not \SIF.
\lipicsEnd
\end{remark}

\section{Occurrence-Preserving Morphisms}

After  introducing and discussing \IF morphisms, 
we are ready to present the main result,
which establishes sufficient conditions under which factor occurrences are preserved.

\begin{theorem}\label{thm:occ-preserve}
Let $\phi: \Sigma^* \to \Gamma^*$ be an injective morphism,
let $u, v \in \Sigma^*$ be two words,
and let $k \geq 1$ be an integer.
If $\phi$ is interference-free on $\bigcup_{i=0}^{k-1} \{ \phi^i(u) \}$,
then, 
\[
\occ{u}[v] = \occ{\phi^k(u)}[\phi^k(v)].
\]
\end{theorem}

\noindent
To prove this theorem, we first observe that it suffices to establish the case $k=1$.
The general case then follows by induction on $k$, using an argument analogous to that for the base case $k=1$.
Accordingly, we aim to show the following.

\begin{proposition}\label{prop:count-invariance}
Let $\phi: \Sigma^* \to \Gamma^*$ be an injective morphism and
let $u, v \in \Sigma^*$ be two words.
If $\phi$ is \IF on $\{u\}$,
then $\occ{u}[v] = \occ{\phi(u)}[\phi(v)]$.
\end{proposition}

We begin with the following lemma,
which intuitively states that applying a non-erasing morphism cannot ``eliminate'' existing factor occurrences.

\begin{lemma}\label{lem:count:leq}
Let $\phi: \Sigma^* \to \Gamma^*$ be a non-erasing morphism and
let $u, v \in \Sigma^*$ be two words. Then
$\occ{u}[v] \leq \occ{\phi(u)}[\phi(v)]$.
\end{lemma}

\begin{proof}
For each occurrence of $u$ in $v$,
there exist words $\pref{v}, \suf{v} \in \Sigma^*$ such that $v = \pref{v} \cdot u \cdot \suf{v}$.
Applying $\phi$ to both sides gives
$\phi(v) = \phi(\pref{v}) \cdot \phi(u) \cdot \phi(\suf{v})$.
Hence, each occurrence of $u$ in $v$ yields an occurrence of $\phi(u)$ in $\phi(v)$.
Moreover, distinct occurrences of $u$ in $v$ yield distinct occurrences of $\phi(u)$ in $\phi(v)$, 
since they correspond to different prefixes $\pref{v}$. %
Therefore, $\occ{u}[v] \leq \occ{\phi(u)}[\phi(v)]$.
\end{proof}

Note that the non-erasing condition is not explicitly stated in \cref{thm:occ-preserve};
the following observation shows that it is already implied by injectivity.

\begin{observation}
Let $\phi: \Sigma^* \to \Gamma^*$ be a morphism.
If $\phi$ is injective then $\phi$ is non-erasing.
\end{observation}

\begin{proof}
We proceed by contrapositive.
If $\phi$ is erasing, then there exists $c \in \Sigma$ such that $\phi(c) = \varepsilon$.
Then, $\phi(cc) = \phi(c) = \varepsilon$ but $cc \neq c$, which means $\phi$ is not injective.
\end{proof}

\begin{figure}[t]
\centering
\includegraphics[width=\linewidth]{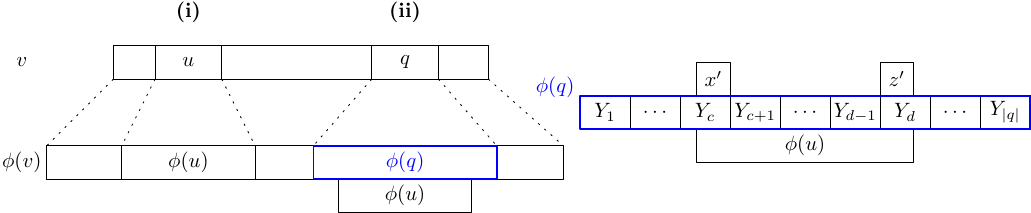}
\caption{Left: illustration of \cref{lem:good-bad-cases}. Right: illustration of the proof of \cref{prop:count-invariance}.}
\label{fig:good-bad-cases}
\end{figure}

The next lemma traces occurrences of $\phi(u)$ in $\phi(v)$ back to
the corresponding occurrences  in $v$,
identifying the only two possible cases, illustrated on the left of \cref{fig:good-bad-cases}.

\begin{lemma}\label{lem:good-bad-cases}
Let $\phi: \Sigma^* \to \Gamma^*$ be an injective morphism and
let $u, v \in \Sigma^*$ be two words.
If $\phi(u)$ is a factor of $\phi(v)$, then at least one of the following two cases holds:
\begin{enumerate}[(i)] %
\item\label{good-case} 
$u$ is a factor of  $v$, or
\item\label{bad-case}
there exists a factor $q$ of $v$ such that 
$q \neq  u$ and $\phi(u)$ is a proper factor of $\phi(q)$.
\end{enumerate}
\end{lemma}

\begin{proof}
Assume, for contradiction, that $\phi(u)$ is a factor of $\phi(v)$ while 
neither Case~(\ref{good-case}) nor Case~(\ref{bad-case}) holds.
Since $u$ is not a factor of $v$,
the only possible way for $\phi(u)$ to be a factor of $\phi(v)$
is that there exists a factor $q'$ of $v$ such that $q' \neq u$ and $\phi(q') = \phi(u)$.
However, this contradicts the injectivity of $\phi$.
\end{proof}

\begin{remark}
In the context of \cref{prop:count-invariance},
Case~(\ref{good-case}) corresponds to the situation in which the occurrence count is preserved.
In contrast, Case~(\ref{bad-case}) 
corresponds to the situation in which ``interference'' causes an occurrence-count violation. %
In \cref{lem:good-bad-cases}, only the injectivity of $\phi$ is assumed.
In the proof of \cref{prop:count-invariance}, we show that imposing the \IF condition rules out the undesirable Case~(\ref{bad-case}).
\lipicsEnd
\end{remark}

\begin{proof}[Proof of \cref{prop:count-invariance}]
We aim to show that 
$\occ{u}[v] \geq \occ{\phi(u)}[\phi(v)]$.
We follow the two cases described in \cref{lem:good-bad-cases}.
In particular, we prove that Case~(\ref{bad-case}) is impossible and 
only Case~(\ref{good-case}) is possible under the \IF condition in \cref{thm:occ-preserve}.
Recall that $q$ was introduced in \cref{lem:good-bad-cases}.
First observe that if $|q| = 1$, i.e., $q$ is a symbol,  
then Case~(\ref{bad-case}) is impossible, 
as it would force one of the following two contradictions.
For any $1 \leq i \leq |v|$, 
\begin{itemize}
\item 
if $\phi(u)$ is a proper prefix or a proper suffix of $\phi(v[i])$, 
then $\phi(u)$ admits an interfered image factorization where 
$x = y = \varepsilon$ and $z$ is a proper prefix of $\phi(v[i])$, or 
$x$ is a proper suffix of $\phi(v[i])$ and $y = z = \varepsilon$;
\item 
if $\phi(u)$ is a proper factor of $\phi(v[i])$ but is neither a prefix nor a suffix of $\phi(v[i])$, then $\phi(u)$ is an inner image factor.
\end{itemize}
Thus, we know that $|q|>1$.
We now show that if $\phi$ is \IF on $\{u\}$ 
then Case~(\ref{bad-case}) is impossible.
We proceed by contrapositive: assume  Case~(\ref{bad-case}) does occur, and we will show that $\phi$ is not \IF on $\{u\}$.
Let $\phi(q) = Y_1 \cdots Y_{|q|}$ where  $Y_{i} \in \images{\phi}$ for $1 \leq i \leq |q|$.
If $\phi(u)$ is a proper factor of $\phi(q)$,
then there exists a factor of $\phi(q)$ of the form  $Y_{c} \cdots Y_{d}$
(for some $1 \leq c < d \leq |q|$)  such that 
$\phi(u)$ is a proper factor of $Y_{c} \cdots Y_{d}$ and 
$Y_{c+1} \cdots Y_{d-1}$ is a proper factor of $\phi(u)$.
(We define $Y_{c+1} \cdots Y_{d-1}$ to be $\varepsilon$ when $d = c + 1$.)
Hence, we have 
$\phi(u) = x' \cdot Y_{c+1} \cdots Y_{d-1} \cdot z'$, 
where $x'$ is a suffix of $Y_c$, 
$z'$ is a prefix of $Y_d$, and 
$x' \cdot z' \neq Y_c \cdot Y_d$ (illustrated on the right of \cref{fig:good-bad-cases}).
Now we are ready to show that $\phi$ is not \IF on $\{u\}$ by
constructing an interfered image factorization of 
 $\phi(u) = x \cdot y \cdot z$,
where
\[
x = 
\begin{cases}
\varepsilon \text{~if~} x' = Y_c, \\
x' \text{~otherwise};
\end{cases}
\hspace{-0.8em}
z =
\begin{cases}
\varepsilon \text{~if~} z' = Y_d, \\
z' \text{~otherwise};
\end{cases}
\hspace{-0.8em}
y =
\begin{cases}
Y_{c+1} \cdots Y_{d-1} \text{~if~} x = x' \text{~and~} z = z', \\
Y_{c} \cdots Y_{d-1} \text{~if~} x = \varepsilon \text{~and~} z = z', \\
Y_{c+1} \cdots Y_{d} \text{~if~} x = x' \text{~and~} z = \varepsilon.
\end{cases}
\]
Thus, we have proved by contrapositive that Case~(\ref{bad-case}) is impossible 
when $\phi$ is \IF on $\{u\}$.
Therefore, only Case~(\ref{good-case}) is possible under the conditions in \cref{thm:occ-preserve}.
In other words, we have shown that
if $\phi(u)$ is a factor of $\phi(v)$,
then $u$ is a factor of $v$.
This implies that $\occ{u}[v] \geq \occ{\phi(u)}[\phi(v)]$.
Combining with \cref{lem:count:leq}, we conclude that
$\occ{u}[v] = \occ{\phi(u)}[\phi(v)]$.
\end{proof}

Finally, \cref{thm:occ-preserve} follows from \cref{prop:count-invariance} by induction.

\section{Applications of Occurrence-Preserving Morphisms}

In this section, we present several applications of the characterization of occurrence-preserving morphisms established in the previous section. 
As an immediate consequence of \cref{thm:occ-preserve}, 
occurrence preservation holds at a more granular level:
not only is the number of occurrences preserved, but there also exists a bijection
between the starting positions of the occurrences of $u$ in $v$ and those of $\phi^k(u)$ in $\phi^k(v)$.

\begin{lemma}
Let $\phi: \Sigma^* \to \Gamma^*$ be an injective morphism,
let $u, v \in \Sigma^*$ be two words,
and let $k \geq 1$ be an integer.
If $\phi$ is interference-free on $\bigcup_{i=0}^{k-1} \{ \phi^i(u) \}$,
then,
$p \in \Occ{u}[v]$ if and only if $|\phi^k(v[1 \ldots p - 1])|+1 \in \Occ{\phi^k(u)}[\phi^k(v)]$.
\end{lemma}

\begin{proof}
$(\Rightarrow)$
If $u$ occurs at position $p$ in $v$, 
then $v$ can be factorized as $v = v[1 \ldots p -1] \cdot u \cdot y$
for $y \in \Sigma^*$.
It follows that
$\phi^k(v) = \phi^k(v[1 \ldots p -1]) \cdot \phi^k(u) \cdot \phi^k(y)$,
 and thus $\phi^k(u)$ occurs at position $|\phi^k(v[1 \ldots p -1])|+1$ in $\phi^k(v)$.

$(\Leftarrow)$
From the ``$\Rightarrow$'' direction 
we can infer that  $\occ{u}[v] \leq \occ{\phi^k(u)}[\phi^k(v)]$.
By \cref{thm:occ-preserve}, we know that 
$\occ{u}[v] = \occ{\phi^k(u)}[\phi^k(v)]$.
Thus, there does not exist an occurrence $p'$ of $\phi^k(u)$
in $\phi^k(v)$ such that 
$p' \neq |\phi^k(v[1 \ldots p -1])|+1$
for any $p \in \Occ{u}[v]$.
Therefore, the ``$\Leftarrow$'' direction also holds.
\end{proof}

Note that if $\phi$ is $\ell$-uniform,
then $|\phi^k(v[1 \ldots p -1])|+1 = \ell^k (p-1)+1$.

\subsection*{MUSs of Fibonacci and Thue-Morse Words} %

In this subsection, as another application of \cref{thm:occ-preserve},
we identify the minimal unique substrings (MUSs) of Fibonacci and Thue-Morse words.

\subparagraph*{MUSs of Fibonacci words.}

We first state two observations on $F_i^{\lhd} = F_i[1 \ldots f_i-1]$.

\begin{observation}\label{lem:occ-lpp-fk} %
For $i \geq k \geq 4$,
$\occ{F_k^\lhd}[F_i] = \occ{F_k}[F_i]$.
\end{observation}

\begin{proof}
First observe that each occurrence of $F_k^\lhd$ is followed by either $\a$ or $\b$.
Hence 
$\occ{F_k^\lhd}[F_i]  = \occ{F_k^\lhd\cdot \a}[F_i] + \occ{F_k^\lhd\cdot \b}[F_i]$.
By \cref{obs:fib-properties},
when $k$ is even, $F_k$ ends with $\a\b\a$,
so $F_k^\lhd$ ends with $\a\b$ and $F_k^\lhd \cdot \a = F_k$.
Further, by \cref{obs:fib-properties}, $\b\b$ does not occur in $F_i$,
which means $\occ{F_k^\lhd\cdot \b}[F_i] = 0$.
Thus, $\occ{F_k^\lhd}[F_i] 
= \occ{F_k^\lhd\cdot \a}[F_i] = \occ{F_k}[F_i]$.
The case where $k$ is odd can be proved analogously using the facts that $\a\a\a$ does not occur in $F_i$ and $F_k$ ends with $\a\a\b$.
\end{proof}

\begin{observation}\label{obs:if-lhd}
$\varphi$ is \IF on $\{ F_{i}^{\lhd} :  i \geq 4  \text{~and~} i \text{~is odd}\}$.
\end{observation}

\begin{proof}
Since $i$ is odd, $F_i$ ends with $\a\b$ and $F_{i+1}$ ends with $\b\a$ by \cref{obs:fib-properties}.
With $F_i = F_i^\lhd \cdot \b$,
applying $\varphi$, we obtain
$F_{i+1}=\varphi(F_i)=\varphi(F_i^\lhd)\varphi(\b)=\varphi(F_i^\lhd)\cdot \a$.
Since $\varphi(\b) = \a$,  we have $\varphi(F_{i}^{\lhd}) = F_{i+1}^{\lhd}$.
The claim then follows using a similar argument for \cref{cor:fib-if-even}
and the fact that $\a\b$ is both a prefix and a suffix of $\varphi(F_{i}^{\lhd})$.
\end{proof}

Occurrences of smaller-order Fibonacci words 
(i.e., occurrences of $F_{k}$ in $F_{i}$ for $k < i$) 
have been studied previously~\cite{conf/cpm/2025/guo}.
Using our characterization of occurrence-preserving morphisms, we show that $\occ{F_{i-d}}[F_i]$ is independent of $i$ once the offset $d$ is fixed.

\begin{lemma}\label{lem:occ-f-i-d}
For $i,j,d$ with $i \geq j$ and $j-d \geq 4$,
$\occ{F_{i-d}}[F_i] = \occ{F_{j-d}}[F_{j}]$.
\end{lemma}

\begin{proof}
It suffices to prove the case $j=i-1$,
i.e., $\occ{F_{i-d}}[F_i] = \occ{F_{i-1-d}}[F_{i-1}]$,
and the result then follows.
When $i-d$ is even, the desired equality holds by \cref{cor:fib-if-even} and \cref{thm:occ-preserve}.
For odd $i-d \geq 4$, we have
\[
\occ{F_{i-d}}[F_i] 
= \occ{F_{i-d}^{\lhd}}[F_i] 
= \occ{F_{i-1-d}^{\lhd}}[F_{i-1}] 
= \occ{F_{{i-1-d}}}[F_{i-1}],
\]
where the first and last equalities hold by \cref{lem:occ-lpp-fk}, while the second equality holds by \cref{obs:if-lhd} and \cref{thm:occ-preserve}.
This completes the proof.
\end{proof}

\begin{figure}[t]
\centering
\includegraphics[width=0.55\linewidth]{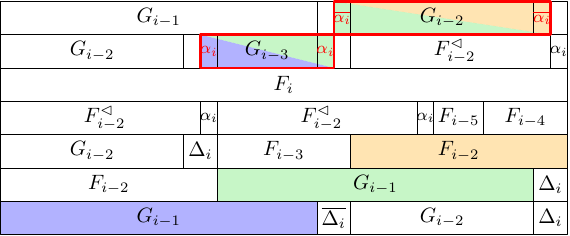}
\caption{
Illustration of \cref{thm:mus-fib}.
In the top two factorizations of $F_i$, each MUS is highlighted in a red rectangle.
In the bottom three factorizations of $F_i$, each net occurrence is colored.
Moreover,  the colors of the two consecutive net occurrences
together are used as the colors to 
highlight the occurrence of the MUS they correspond to. 
The uncolored factorization is used in the proof.
}
\label{fig:fib-mus-no}
\end{figure}

\begin{theorem}\label{thm:mus-fib}
For each $i \geq 6$, let $\alpha_i = \a$ if $i$ is even, and $\alpha_i = \b$ if $i$ is odd. Then,
$
\mus{F_i} = \{  \alpha_i  \, G_{i-3} \, \alpha_i,  \, \flip{\alpha_i} \, G_{i-2}  \, \flip{\alpha_i}   \}
$.
\end{theorem}

\begin{proof}
We proceed to show the set equality by proving 
subset relations in both directions.

We first show that 
$\flip{\alpha_i} \, G_{i-2} \, \flip{\alpha_i} \in \mus{F_i}$.
Its unique occurrence is outlined in red in \cref{fig:fib-mus-no}.
By \cref{lem:occ-f-i-d}, we have $\occ{F_{i-2}}[F_i] = \occ{F_{4}}[F_6] = 3$.
The factorizations 
$F_i = F_{i-2}  F_{i-2}  F_{i-5}   F_{i-4}$ and 
$F_i = F_{i-2}  F_{i-3}  F_{i-2}$
reveal all three occurrences of $F_{i-2}$.
(In \cref{fig:fib-mus-no}, these two factorizations correspond to
$F_i = F_{i-2}^\lhd \alpha_i  F_{i-2}^\lhd \alpha_i  F_{i-5}   F_{i-4}$ and 
$F_i = G_{i-2} \Delta_i  F_{i-3}  F_{i-2}$.)
Moreover, by \cref{lem:occ-lpp-fk}, we have $\occ{F_{i-2}^{\lhd}}[F_i] = \occ{F_{i-2}}[F_i] = 3$.
Among the three occurrences of $F_{i-2}^{\lhd} = G_{i-2} \, \flip{\alpha_i}$ in $F_{i}$, 
exactly one is preceded by $\flip{\alpha_i}$ 
(namely the one corresponding to the occurrence of $F_{i-2}$ highlighted in yellow in \cref{fig:fib-mus-no}).
Hence, $\flip{\alpha_i} \, G_{i-2} \, \flip{\alpha_i}$ is unique in $F_i$.
Further, $\flip{\alpha_i} \,G_{i-2} $ is not unique in $F_i$, 
since it is a substring of $G_{i-1}$, which occurs twice in $F_i$ (by \cref{obs:fib-properties}).
Now, since $\flip{\alpha_i} \, G_{i-2} \, \flip{\alpha_i}$ is unique, 
while both $G_{i-2} \, \flip{\alpha_i}$ and $\flip{\alpha_i} \,G_{i-2}$ are repeated,
we have shown that $\flip{\alpha_i} \, G_{i-2} \, \flip{\alpha_i}$ is a MUS.
By an analogous argument, 
$\alpha_i \, G_{i-3} \, \alpha_i \in \mus{F_i}$ can be proved similarly. 

We next show $ \mus{F_i} \subset \{ \alpha_i  \, G_{i-3} \, \alpha_i,  \ \flip{\alpha_i} \, G_{i-2} \, \flip{\alpha_i}  \}$. 
Let $[s, e]$  be an occurrence of a substring $P = F_i[s\ldots e]$ of $F_i$ such that
$[s, e]$ is not a proper sub-occurrence of
$[f_{i-2}, f_{i-1}-1]$ or $[f_{i-1}, f_i-1]$,
which are the occurrences of $\alpha_i  \, G_{i-3} \, \alpha_i $ and $ \flip{\alpha_i} \, G_{i-2} \, \flip{\alpha_i}$, respectively.
We claim that $P$ is repeated and thus cannot be a MUS.
We consider the following cases.
\begin{itemize}
\item 
When $e \leq f_{i-1}-2$, $P$ is contained in the repeated substring $F_i[1 \ldots f_{i-1}-2] = G_{i-1}$,
highlighted in blue in \cref{fig:fib-mus-no}.
\item 
When $f_{i-2}+1 \leq s \leq f_{i-1}-1$ and $f_{i-1} \leq e \leq f_i-2$, 
$P$ is a contained in the repeated substring $F_i[f_{i-2}+1 \ldots f_i-2] = G_{i-1}$,
highlighted in green in \cref{fig:fib-mus-no}.
\item
When $f_{i-1}+1 \leq s$ and $e = f_i$, 
$P$ is contained in the repeated substring $F_i[f_{i-1}+ 1 \ldots f_{i}]  = F_{i-2}$,
highlighted in yellow in \cref{fig:fib-mus-no}.
\end{itemize}
In all cases, $P$ is repeated in $F_i$ and thus cannot be a MUS.
Therefore, we conclude that 
$
\mus{F_i} = \{  \alpha_i  \, G_{i-3} \, \alpha_i,  \, \flip{\alpha_i} \, G_{i-2}  \, \flip{\alpha_i}   \}
$.
\end{proof}

\subparagraph*{MUSs of Thue-Morse words.}
Before proving the MUSs of Thue-Morse words, we first state the following lemma, 
which serves as the Thue-Morse analogue of \cref{lem:occ-f-i-d}.

\begin{lemma}\label{lem:occ-tm-i-3-tm-5}
For each $i \geq j$ and $j - d \geq 2$, we have 
$\occ{\tm_{i-d}}[\tm_{i}] = \occ{\tm_{j-d}}[\tm_{j}]$
and 
$\occ{\flip{\tm_{i-d}}}[\tm_{i}] = \occ{\flip{\tm_{j-d}}}[\tm_{j}]$.
\end{lemma}

\begin{proof}
The first equality follows directly from \cref{cor:tm-if} and \cref{thm:occ-preserve}.
For the second equality,  note that $\mu$ is \IF on $\{ \flip{\tm_i} : i \geq 4 \}$, as can be shown
by an argument analogous to \cref{cor:tm-if};
combining this with \cref{thm:occ-preserve} yields the desired result.
\end{proof}

\begin{figure}[t]
\centering
\includegraphics[width=\linewidth]{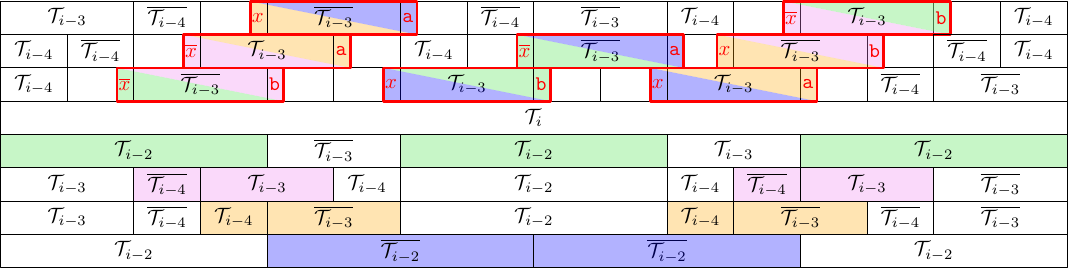}
\caption{
Illustration of \cref{thm:mus-tm},
following the conventions of \cref{fig:fib-mus-no}.
Let $x = \b$ for even $i$
and $x = \a$ for odd $i$.
}
\label{fig:mus-tm}
\end{figure}

\begin{theorem}\label{thm:mus-tm}
For a word $w$, let $\ext{w} = \{ 
\a \, w \, \a,
\a \, w \, \b,
\b \, w \, \a,
\b \, w \, \b
\}$
be the set of extensions of $w$.
Then, 
for each $i \geq 5$,
$
\mus{\tm_i} = \ext{\tm_{i-3}} \cup \ext{\flip{\tm_{i-3}}}
$.
\end{theorem}

\begin{proof}
As in the proof of \cref{thm:mus-fib}, 
we establish the set equality by proving both inclusions.
We first show that $\ext{\tm_{i-3}} \subset \mus{\tm_i}$.
By \cref{lem:occ-tm-i-3-tm-5},
we have 
$\occ{\tm_{i-3}}[\tm_{i}] = \occ{\tm_2}[\tm_{5}] = 4$.
Consider the four occurrences of $\tm_{i-3}$ in $\tm_{i}$. 
Each occurrence is preceded and followed by either $\tm_{i-4}$ or $\flip{\tm_{i-4}}$.
More precisely, across the four occurrences, all four possible combinations occur:
$\tm_{i-4}  \tm_{i-3} \tm_{i-4}$,
$\flip{\tm_{i-4}} \tm_{i-3} \flip{\tm_{i-4}}$,
$\tm_{i-4} \tm_{i-3} \flip{\tm_{i-4}}$, and
$\flip{\tm_{i-4}} \tm_{i-3} \tm_{i-4}$.
Since $\tm_{i-4}$ and $\flip{\tm_{i-4}}$ have different first and last letters,
It follows that each $u \in \ext{\tm_{i-3}}$ is unique.
Further, since 
$\tm_{i-4}  \tm_{i-3}$,
$\flip{\tm_{i-4}}  \tm_{i-3}$,
$\tm_{i-3}  \flip{\tm_{i-4}}$, and
$\tm_{i-3}  \tm_{i-4}$
all occur twice in $\tm_i$,
it follows that 
$\a \, \tm_{i-3}$,
$\b \, \tm_{i-3}$,
$\tm_{i-3} \, \a$, and
$\tm_{i-3} \, \b$
also all occur twice in $\tm_i$.
Therefore, $\ext{\tm_{i-3}} \subset \mus{\tm_i}$.
With a similar argument, it can be shown that 
$\ext{\flip{\tm_{i-3}}} \subset \mus{\tm_i}$.
All these eight MUSs are outlined in red in \cref{fig:mus-tm}.

We next prove $\mus{\tm_i} \subset \ext{\tm_{i-3}} \cup \ext{\flip{\tm_{i-3}}}$.
Let $[s, e]$ be an occurrence of a substring $P = \tm_i[s\ldots e]$ of $\tm_i$
such that $[s, e]$ is not a proper sub-occurrence of any  
occurrences of substrings in $\ext{\tm_{i-3}} \cup \ext{\flip{\tm_{i-3}}}$. 
We claim that $P$ is repeated and thus cannot be a MUS.
We consider the following cases.
\begin{itemize}
\item
When $s \leq \ltm_{i-3}-1$ and $e \leq \ltm_{i-2}$,
$P$ is contained in the repeated substring $\tm_i[1 \ldots \ltm_{i-2}] = \tm_{i-2}$, 
highlighted in green on the left of \cref{fig:mus-tm}.

\item
When $\ltm_{i-3}+1 \leq s \leq \ltm_{i-3}+\ltm_{i-4}-1$
and $\ltm_{i-2}+1 \leq e \leq \ltm_{i-2}+\ltm_{i-4}$,
$P$ is contained in the repeated substring 
$\tm_i[\ltm_{i-3}+1 \ldots \ltm_{i-2}+\ltm_{i-4}] = \flip{\tm_{i-4}}\,\tm_{i-3}$, 
highlighted in pink on the left of \cref{fig:mus-tm}.

\item
When $\ltm_{i-3}+\ltm_{i-4}+1 \leq s \leq \ltm_{i-2}-1$
and $\ltm_{i-2}+\ltm_{i-4}+1 \leq e \leq \ltm_{i-2}+\ltm_{i-3}$,
$P$ is contained in the repeated substring 
$\tm_i[\ltm_{i-3}+\ltm_{i-4}+1 \ldots \ltm_{i-2}+\ltm_{i-3}] = \tm_{i-4}\,\flip{\tm_{i-3}}$, 
highlighted in yellow on the left of \cref{fig:mus-tm}.

\item
When $\ltm_{i-2}+1 \leq s \leq \ltm_{i-2}+\ltm_{i-3}-1$
and $\ltm_{i-2}+\ltm_{i-3}+1 \leq e \leq \ltm_{i-1}$,
$P$ is contained in the repeated substring 
$\tm_i[\ltm_{i-2}+1 \ldots \ltm_{i-1}] = \flip{\tm_{i-2}}$, 
highlighted in blue on the left of \cref{fig:mus-tm}.

\item
When $\ltm_{i-2}+\ltm_{i-3}+1 \leq s \leq \ltm_{i-1}-1$
and $\ltm_{i-1}+1 \leq e \leq \ltm_{i-1}+\ltm_{i-3}$,
$P$ is contained in the repeated substring 
$\tm_i[\ltm_{i-2}+\ltm_{i-3}+1 \ldots \ltm_{i-1}+\ltm_{i-3}] = \tm_{i-2}$, 
highlighted in green in the middle of \cref{fig:mus-tm}.

\end{itemize}
The remaining four cases can be proved analogously, and are illustrated by the blue, yellow, pink, and green substrings on the right of \cref{fig:mus-tm}.
In all cases, $P$ is repeated in $\tm_i$ and thus cannot be a MUS.
Therefore, we conclude that 
$
\mus{\tm_i} = \ext{\tm_{i-3}} \cup \ext{\flip{\tm_{i-3}}}
$.
\end{proof}

Finally, with the two theorems above, we can 
exploit the connection between MUSs and net occurrences (\cref{lem:no-mus}) to simplify existing proofs on net occurrences in these words~\cite{conf/cpm/2025/guo} (\cref{lem:no-fib-tm}).
These results are also illustrated in \cref{fig:fib-mus-no} and \cref{fig:mus-tm}.

\begin{lemma}[\cite{conf/cpm/2025/mieno, journal/fuin/2011/ilie}]\label{lem:no-mus}
For a string $w$, let 
$[i_1, j_1], [i_2, j_2],   \ldots,  [i_m, j_m] $ be the sequence of MUS occurrences in $w$, 
ordered by increasing starting position.
Then, $\no{w} = \{ [1, j_1-1], [i_1+1, j_2-1], \ldots, [i_{m-1}+1, j_m-1], [i_m+1, |w|] \} $.
\end{lemma}

\begin{corollary}[\cite{conf/cpm/2025/guo}]\label{lem:no-fib-tm}
For each $i \geq 7$, 
$\no{F_i}$ consists of the two occurrences of $G_{i-1}$ and the last occurrence of $F_{i-2}$.
For each $i \geq 5$, $\no{\tm_i}$ consists of all occurrences of each of the following substrings:
$\tm_{i-2}$,
$\flip{\tm_{i-2}}$,
$\tm_{i-4} \ \flip{\tm_{i-3}}$, and
$\flip{\tm_{i-4}} \ \tm_{i-3}$.
\end{corollary}

\section{Conclusion and Future Work}

In this work, we investigated morphisms that preserve factor occurrences.  
We introduced the notion of \emph{interference-free morphisms} and studied their key properties.
Building on this notion, we established sufficient conditions for \emph{occurrence-preserving morphisms} and
applied these conditions to identify the MUSs of the Fibonacci and Thue-Morse words.
We now outline several directions for future work.

First, one could consider improving the conditions in \cref{thm:occ-preserve} 
so that they are both sufficient and necessary.
A key observation is that interference-freeness is defined with respect to a morphism $\phi$ and a single word $u$,
whereas occurrence-preservation concerns $\phi$ and two words $u$ and $v$.
Thus, extending the definition of interference-freeness to involve the two words may lead to a tighter characterization.
(A natural starting point would be to restrict attention to cases where $\occ{u}[v] > 0$,
since when $\occ{u}[v] = 0$, occurrences can be trivially preserved even if interference-freeness does not hold.)
Furthermore, 
it would be interesting to examine the class of morphisms
that lie between recognizable and \SIF morphisms in the inclusion hierarchy (\cref{remark:hierarchy}).
Another natural direction is to investigate the class of infinite words generated as fixed points of \SIF morphisms.

A second line of future work is algorithmic.
For example, 
in Section 3.2 of \cite{conf/latin/2002/braga}, a particular class of morphisms %
is used  in an algorithm for the inverse problem of overlap-graph construction.
Since their morphisms are defined for a specific algorithmic setting, 
the assumptions on the morphisms are naturally more restrictive: 
for example, they require the morphisms to be uniform.
Further, their definition explicitly excludes  Case~(\ref{bad-case}) in our \cref{lem:good-bad-cases}; 
in contrast, we do not impose such a restriction directly, but instead capture 
the essential behavior 
through the conditions in our \cref{thm:occ-preserve}. 
It would be worthwhile to explore other algorithmic settings in which interference-free morphisms may prove useful.

\clearpage

\bibliography{references}

@inproceedings{conf/mfcs/2025/FRSU,
  author       = {Gabriele Fici and
                  Giuseppe Romana and
                  Marinella Sciortino and
                  Cristian Urbina},
  editor       = {Pawel Gawrychowski and
                  Filip Mazowiecki and
                  Michal Skrzypczak},
  title        = {Morphisms and {BWT}-Run Sensitivity},
  booktitle    = {50th International Symposium on Mathematical Foundations of Computer
                  Science, {MFCS} 2025, August 25-29, 2025, Warsaw, Poland},
  series       = {LIPIcs},
  volume       = {345},
  pages        = {49:1--49:18},
  publisher    = {Schloss Dagstuhl - Leibniz-Zentrum f{\"{u}}r Informatik},
  year         = {2025},
  doi          = {10.4230/LIPIcs.MFCS.2025.49}
}

@article{journal/tit/2021/navarro,
  author       = {Gonzalo Navarro and
                  Carlos Ochoa and
                  Nicola Prezza},
  title        = {On the Approximation Ratio of Ordered Parsings},
  journal      = {{IEEE} Trans. Inf. Theory},
  volume       = {67},
  number       = {2},
  pages        = {1008--1026},
  year         = {2021},
  doi          = {10.1109/TIT.2020.3042746}
}

@article{journal/etds/2023/beal,
    author={Marie{-}Pierre B{\'{e}}al and
                  Dominique Perrin and
                  Antonio Restivo},
    title={Recognizability of morphisms}, 
    volume={43}, 
    DOI={10.1017/etds.2022.109}, 
    number={11}, 
    journal={Ergod. Th. \& Dynam. Sys.}, 
    year={2023},    
    pages={3578–3602}
}

@inproceedings{conf/latin/2002/braga,
  author       = {Mar{\'{\i}}lia D. V. Braga and
                  Joao Meidanis},
  editor       = {Sergio Rajsbaum},
  title        = {An Algorithm That Builds a Set of Strings Given Its Overlap Graph},
  booktitle    = {{LATIN} 2002: Theoretical Informatics, 5th Latin American Symposium,
                  Cancun, Mexico, April 3-6, 2002, Proceedings},
  series       = {Lecture Notes in Computer Science},
  volume       = {2286},
  pages        = {52--63},
  publisher    = {Springer},
  year         = {2002},
  doi          = {10.1007/3-540-45995-2\_10}
}

@inproceedings{conf/cpm/2016/matsuoka,
  author       = {Yoshiaki Matsuoka and
                  Shunsuke Inenaga and
                  Hideo Bannai and
                  Masayuki Takeda and
                  Florin Manea},
  editor       = {Roberto Grossi and
                  Moshe Lewenstein},
  title        = {Factorizing a String into Squares in Linear Time},
  booktitle    = {27th Annual Symposium on Combinatorial Pattern Matching, {CPM} 2016,
                  June 27-29, 2016, Tel Aviv, Israel},
  series       = {LIPIcs},
  volume       = {54},
  pages        = {27:1--27:12},
  publisher    = {Schloss Dagstuhl - Leibniz-Zentrum f{\"{u}}r Informatik},
  year         = {2016},
  doi          = {10.4230/LIPICS.CPM.2016.27}
}

@inproceedings{conf/spire/2015/dumitran,
  author       = {Marius Dumitran and
                  Florin Manea and
                  Dirk Nowotka},
  editor       = {Costas S. Iliopoulos and
                  Simon J. Puglisi and
                  Emine Yilmaz},
  title        = {On Prefix/Suffix-Square Free Words},
  booktitle    = {String Processing and Information Retrieval - 22nd International Symposium,
                  {SPIRE} 2015, London, UK, September 1-4, 2015, Proceedings},
  series       = {Lecture Notes in Computer Science},
  volume       = {9309},
  pages        = {54--66},
  publisher    = {Springer},
  year         = {2015},
  doi          = {10.1007/978-3-319-23826-5\_6}
}

@inproceedings{conf/focs/1982/fredman,
  author       = {Michael L. Fredman and
                  J{\'{a}}nos Koml{\'{o}}s and
                  Endre Szemer{\'{e}}di},
  title        = {Storing a Sparse Table with {O(1)} Worst Case Access Time},
  booktitle    = {23rd Annual Symposium on Foundations of Computer Science, Chicago,
                  Illinois, USA, 3-5 November 1982},
  pages        = {165--169},
  publisher    = {{IEEE} Computer Society},
  year         = {1982},
  doi          = {10.1109/SFCS.1982.39}
}

@inproceedings{conf/dlt/2022/frosini,
  author       = {Andrea Frosini and
                  Ilaria Mancini and
                  Simone Rinaldi and
                  Giuseppe Romana and
                  Marinella Sciortino},
  editor       = {Volker Diekert and
                  Mikhail V. Volkov},
  title        = {Logarithmic Equal-Letter Runs for {BWT} of Purely Morphic Words},
  booktitle    = {Developments in Language Theory - 26th International Conference, {DLT}
                  2022, Tampa, FL, USA, May 9-13, 2022, Proceedings},
  series       = {Lecture Notes in Computer Science},
  volume       = {13257},
  pages        = {139--151},
  publisher    = {Springer},
  year         = {2022},
  doi          = {10.1007/978-3-031-05578-2\_11}
}

@inproceedings{conf/cpm/2023/fici,
  author       = {Gabriele Fici and
                  Giuseppe Romana and
                  Marinella Sciortino and
                  Cristian Urbina},
  editor       = {Laurent Bulteau and
                  Zsuzsanna Lipt{\'{a}}k},
  title        = {On the Impact of Morphisms on {BWT}-Runs},
  booktitle    = {34th Annual Symposium on Combinatorial Pattern Matching, {CPM} 2023,
                  June 26-28, 2023, Marne-la-Vall{\'{e}}e, France},
  series       = {LIPIcs},
  volume       = {259},
  pages        = {10:1--10:18},
  publisher    = {Schloss Dagstuhl - Leibniz-Zentrum f{\"{u}}r Informatik},
  year         = {2023},
  doi          = {10.4230/LIPICS.CPM.2023.10}
}

@inproceedings{conf/iwoca/2019/brlek,
  author       = {Srecko Brlek and
                  Andrea Frosini and
                  Ilaria Mancini and
                  Elisa Pergola and
                  Simone Rinaldi},
  editor       = {Charles J. Colbourn and
                  Roberto Grossi and
                  Nadia Pisanti},
  title        = {{B}urrows-{W}heeler {T}ransform of Words Defined by Morphisms},
  booktitle    = {Combinatorial Algorithms - 30th International Workshop, {IWOCA} 2019,
                  Pisa, Italy, July 23-25, 2019, Proceedings},
  series       = {Lecture Notes in Computer Science},
  volume       = {11638},
  pages        = {393--404},
  publisher    = {Springer},
  year         = {2019},
  doi          = {10.1007/978-3-030-25005-8\_32}
}

@article{journal/dam/1993/berstel,
  author       = {Jean Berstel and
                  Patrice S{\'{e}}{\'{e}}bold},
  title        = {A Characterization of Overlap-Free Morphisms},
  journal      = {Discret. Appl. Math.},
  volume       = {46},
  number       = {3},
  pages        = {275--281},
  year         = {1993},
  doi          = {10.1016/0166-218X(93)90107-Y}
}

@article{hsiao2003square,
  title={Square-free-preserving and primitive-preserving homomorphisms},
  author={Hsiao, H.K. and Yeh, Y.T. and Yu, S.S.},
  journal={Acta Mathematica Hungarica},
  volume={101},
  number={1},
  pages={113--130},
  year={2003},
  publisher={Springer},
  doi = {10.1023/B:AMHU.0000003896.79824.c8}
}

@article{journal/fuin/2022/dolce,
  author       = {Francesco Dolce and
                  Edita Pelantov{\'{a}}},
  title        = {On Morphisms Preserving Palindromic Richness},
  journal      = {Fundam. Informaticae},
  volume       = {185},
  number       = {1},
  pages        = {1--25},
  year         = {2022},
  doi          = {10.3233/FI-222102}
}

@article{journal/dam/1999/richomme,
  author       = {Gw{\'{e}}na{\"{e}}l Richomme and
                  Patrice S{\'{e}}{\'{e}}bold},
  title        = {Characterization of Test-sets for Overlap-free Morphisms},
  journal      = {Discret. Appl. Math.},
  volume       = {98},
  number       = {1-2},
  pages        = {151--157},
  year         = {1999},
  doi          = {10.1016/S0166-218X(99)00118-3}
}

@article{journals/dam/2004/richomme,
  author       = {Gw{\'{e}}na{\"{e}}l Richomme and
                  Francis Wlazinski},
  title        = {Overlap-free morphisms and finite test-sets},
  journal      = {Discret. Appl. Math.},
  volume       = {143},
  number       = {1-3},
  pages        = {92--109},
  year         = {2004},
  doi          = {10.1016/J.DAM.2003.10.005}
}

@article{journal/dmtcs/2007/richomme,
  author       = {Gw{\'{e}}na{\"{e}}l Richomme},
  title        = {On Morphisms Preserving Infinite {L}yndon Words},
  journal      = {Discret. Math. Theor. Comput. Sci.},
  volume       = {9},
  number       = {2},
  year         = {2007},
  doi          = {10.46298/DMTCS.411}
}

@article{journal/afp/2023/holub,
  author       = {Stepan Holub and
                  Martin Raska},
  title        = {Binary codes that do not preserve primitivity},
  journal      = {Arch. Formal Proofs},
  volume       = {2023},
  year         = {2023},
  doi          = {10.1007/s10817-023-09674-2}
}

@article{journal/ijac/1993/carpi,
  author       = {Arturo Carpi},
  title        = {On {A}belian Power-Free Morphisms},
  journal      = {Int. J. Algebra Comput.},
  volume       = {3},
  number       = {2},
  pages        = {151--168},
  year         = {1993},
  url          = {https://doi.org/10.1142/S0218196793000123},
  doi          = {10.1142/S0218196793000123},
  bibsource    = {dblp computer science bibliography, https://dblp.org}
}

@article{journal/dm/2008/huang,
  author       = {C. C. Huang and
                  S. S. Yu},
  title        = {Prefix-primitivity-preserving homomorphisms},
  journal      = {Discret. Math.},
  volume       = {308},
  number       = {7},
  pages        = {1025--1032},
  year         = {2008},
  doi          = {10.1016/J.DISC.2007.03.057}
}

@article{journal/tcs/2009/keranen,
  author       = {Veikko Ker{\"{a}}nen},
  title        = {A powerful abelian square-free substitution over 4 letters},
  journal      = {Theor. Comput. Sci.},
  volume       = {410},
  number       = {38-40},
  pages        = {3893--3900},
  year         = {2009},
  doi          = {10.1016/J.TCS.2009.05.027}
}

@inproceedings{conf/cpm/2025/mieno,
  author       = {Takuya Mieno and
                  Shunsuke Inenaga},
  editor       = {Paola Bonizzoni and
                  Veli M{\"{a}}kinen},
  title        = {Space-Efficient Online Computation of String Net Occurrences},
  booktitle    = {36th Annual Symposium on Combinatorial Pattern Matching, {CPM} 2025,
                  June 17-19, 2025, Milan, Italy},
  series       = {LIPIcs},
  volume       = {331},
  pages        = {23:1--23:13},
  publisher    = {Schloss Dagstuhl - Leibniz-Zentrum f{\"{u}}r Informatik},
  year         = {2025},
  doi          = {10.4230/LIPICS.CPM.2025.23}
}

@article{journal/algorithmica/2022/mieno,
  author       = {Takuya Mieno and
                  Yuta Fujishige and
                  Yuto Nakashima and
                  Shunsuke Inenaga and
                  Hideo Bannai and
                  Masayuki Takeda},
  title        = {Computing Minimal Unique Substrings for a Sliding Window},
  journal      = {Algorithmica},
  volume       = {84},
  number       = {3},
  pages        = {670--693},
  year         = {2022},
  doi          = {10.1007/S00453-021-00864-1}
}

@article{journal/fuin/2011/ilie,
  author       = {Lucian Ilie and
                  William F. Smyth},
  title        = {Minimum Unique Substrings and Maximum Repeats},
  journal      = {Fundam. Informaticae},
  volume       = {110},
  number       = {1-4},
  pages        = {183--195},
  year         = {2011},
  doi          = {10.3233/FI-2011-536}
}

@inproceedings{conf/cpm/2017/mieno,
  author       = {Takuya Mieno and
                  Shunsuke Inenaga and
                  Hideo Bannai and
                  Masayuki Takeda},
  editor       = {Juha K{\"{a}}rkk{\"{a}}inen and
                  Jakub Radoszewski and
                  Wojciech Rytter},
  title        = {Tight Bounds on the Maximum Number of Shortest Unique Substrings},
  booktitle    = {28th Annual Symposium on Combinatorial Pattern Matching, {CPM} 2017,
                  July 4-6, 2017, Warsaw, Poland},
  series       = {LIPIcs},
  volume       = {78},
  pages        = {24:1--24:11},
  publisher    = {Schloss Dagstuhl - Leibniz-Zentrum f{\"{u}}r Informatik},
  year         = {2017},
  doi          = {10.4230/LIPICS.CPM.2017.24}
}

@inproceedings{conf/mfcs/2016/mieno,
  author       = {Takuya Mieno and
                  Shunsuke Inenaga and
                  Hideo Bannai and
                  Masayuki Takeda},
  editor       = {Piotr Faliszewski and
                  Anca Muscholl and
                  Rolf Niedermeier},
  title        = {Shortest Unique Substring Queries on Run-Length Encoded Strings},
  booktitle    = {41st International Symposium on Mathematical Foundations of Computer
                  Science, {MFCS} 2016, August 22-26, 2016 - Krak{\'{o}}w, Poland},
  series       = {LIPIcs},
  volume       = {58},
  pages        = {69:1--69:11},
  publisher    = {Schloss Dagstuhl - Leibniz-Zentrum f{\"{u}}r Informatik},
  year         = {2016},
  doi          = {10.4230/LIPICS.MFCS.2016.69}
}

@inproceedings{conf/spire/2024/ohlebusch,
  author       = {Enno Ohlebusch and
                  Thomas B{\"{u}}chler and
                  Jannik Olbrich},
  editor       = {Zsuzsanna Lipt{\'{a}}k and
                  Edleno Silva de Moura and
                  Karina Figueroa and
                  Ricardo Baeza{-}Yates},
  title        = {Faster Computation of {C}hinese Frequent Strings and Their Net Frequencies},
  booktitle    = {String Processing and Information Retrieval - 31st International Symposium,
                  {SPIRE} 2024, Puerto Vallarta, Mexico, September 23-25, 2024, Proceedings},
  series       = {Lecture Notes in Computer Science},
  volume       = {14899},
  pages        = {249--256},
  publisher    = {Springer},
  year         = {2024},
  doi          = {10.1007/978-3-031-72200-4\_19}
}

@article{DBLP:journals/corr/abs-2410-06837,
  author       = {Shunsuke Inenaga},
  title        = {Faster and Simpler Online Computation of String Net Frequency},
  journal      = {CoRR},
  year         = {2024},
  doi          = {10.48550/arXiv.2410.06837}
}

@inproceedings{conf/spire/2024/guo,
  author       = {Peaker Guo and
                  Seeun William Umboh and
                  Anthony Wirth and
                  Justin Zobel},
  editor       = {Zsuzsanna Lipt{\'{a}}k and
                  Edleno Silva de Moura and
                  Karina Figueroa and
                  Ricardo Baeza{-}Yates},
  title        = {Online Computation of String Net Frequency},
  booktitle    = {String Processing and Information Retrieval - 31st International Symposium,
                  {SPIRE} 2024, Puerto Vallarta, Mexico, September 23-25, 2024, Proceedings},
  series       = {Lecture Notes in Computer Science},
  volume       = {14899},
  pages        = {159--173},
  publisher    = {Springer},
  year         = {2024},
  doi          = {10.1007/978-3-031-72200-4\_12}
}

@inproceedings{conf/cpm/2024/guo,
  author       = {Peaker Guo and
                  Patrick Eades and
                  Anthony Wirth and
                  Justin Zobel},
  editor       = {Shunsuke Inenaga and
                  Simon J. Puglisi},
  title        = {Exploiting New Properties of String Net Frequency for Efficient Computation},
  booktitle    = {35th Annual Symposium on Combinatorial Pattern Matching, {CPM} 2024,
                  June 25-27, 2024, Fukuoka, Japan},
  series       = {LIPIcs},
  volume       = {296},
  pages        = {16:1--16:16},
  publisher    = {Schloss Dagstuhl - Leibniz-Zentrum f{\"{u}}r Informatik},
  year         = {2024},
  doi          = {10.4230/LIPICS.CPM.2024.16}
}

@inproceedings{conf/cpm/2025/guo,
  author       = {Peaker Guo and
                  Kaisei Kishi},
  editor       = {Paola Bonizzoni and
                  Veli M{\"{a}}kinen},
  title        = {Net Occurrences in {F}ibonacci and {T}hue-{M}orse Words},
  booktitle    = {36th Annual Symposium on Combinatorial Pattern Matching, {CPM} 2025,
                  June 17-19, 2025, Milan, Italy},
  series       = {LIPIcs},
  volume       = {331},
  pages        = {16:1--16:22},
  publisher    = {Schloss Dagstuhl - Leibniz-Zentrum f{\"{u}}r Informatik},
  year         = {2025},
  doi          = {10.4230/LIPICS.CPM.2025.16}
}

@article{journal/cacm/1975/aho,
  author       = {Alfred V. Aho and
                  Margaret J. Corasick},
  title        = {Efficient String Matching: An Aid to Bibliographic Search},
  journal      = {Commun. {ACM}},
  volume       = {18},
  number       = {6},
  pages        = {333--340},
  year         = {1975},
  doi          = {10.1145/360825.360855}
}

@book{book/daglib/0025093,
  author       = {Jean Berstel and
                  Dominique Perrin and
                  Christophe Reutenauer},
  title        = {Codes and Automata},
  series       = {Encyclopedia of mathematics and its applications},
  volume       = {129},
  publisher    = {Cambridge University Press},
  year         = {2010},
  isbn         = {978-0-521-88831-8}
}

@incollection{book/Rozenberg1976,
author="Rozenberg, G.
and Salomaa, A.",
editor="Tou, Julius T.",
title="The Mathematical Theory of {L} Systems",
bookTitle="Advances in Information Systems Science: Volume 6",
year="1976",
publisher="Springer US",
address="Boston, MA",
pages="161--206",
isbn="978-1-4615-8249-6",
doi="10.1007/978-1-4615-8249-6_4"
}

@article{journal/tcs/2025/navarrourbina,
title = {Repetitiveness measures based on string morphisms},
journal = {Theoretical Computer Science},
volume = {1043},
pages = {115259},
year = {2025},
issn = {0304-3975},
doi = {https://doi.org/10.1016/j.tcs.2025.115259},
author = {Gonzalo Navarro and Cristian Urbina}
}

@inproceedings{conf/cpm/2026/KI,
  author       = {Kotaro Kimura and
                  Tomohiro I},
  editor       = {Philip Bille and
                  Nicola Prezza},  
  title        = {R-enum Revisited: Speedup and Extension for Context-Sensitive Repeats
                  and Net Frequencies},
  booktitle    = {37th Annual Symposium on Combinatorial Pattern Matching, {CPM} 2026,
                  June 15-17, 2026, Copenhagen, Denmark},
  series       = {LIPIcs},
  volume       = {},
  pages        = {},
  publisher    = {Schloss Dagstuhl - Leibniz-Zentrum f{\"{u}}r Informatik},
  year         = {2026},
  doi          = {10.48550/arXiv.2511.11057}
}

@inproceedings{conf/mfcs/1993/berstel,
  author       = {Jean Berstel and
                  Patrice S{\'{e}}{\'{e}}bold},
  editor       = {Andrzej M. Borzyszkowski and
                  Stefan Sokolowski},
  title        = {A Characterization of {S}turmian Morphisms},
  booktitle    = {Mathematical Foundations of Computer Science 1993, 18th International
                  Symposium, MFCS'93, Gdansk, Poland, August 30 - September 3, 1993,
                  Proceedings},
  series       = {Lecture Notes in Computer Science},
  volume       = {711},
  pages        = {281--290},
  publisher    = {Springer},
  year         = {1993},
  doi          = {10.1007/3-540-57182-5\_20}
}

@misc{oeis,
  author = {{OEIS Foundation Inc.}},
  title  = {The {O}n-{L}ine {E}ncyclopedia of {I}nteger {S}equences},
  url    = {https://oeis.org},
  note   = {Founded by Neil J. A. Sloane}
}

\appendix

\section{An Algorithm for Deciding Interference-Freeness}\label{appendix:algo}

In this section, we present the detailed proofs of \cref{thm:if-algo}.
The main algorithm is presented in \cref{algo:if-check}.
We begin by introducing the data structure used in the algorithm.

An \emph{Aho-Corasick automaton} is a trie over a set of strings,
augmented by \emph{failure} and \emph{output} links.
For a node $q$ in the trie, let $\str{q}$ denote the concatenation of the edge labels 
on the path from the root to $q$.
There is a failure link from node $q$ to node $v$
if $\str{v}$ is the longest proper suffix of $\str{q}$ that is also a prefix of some pattern in the dictionary.

We next prove the following two lemmas to establish \cref{thm:if-algo}.

\begin{lemma}\label{lem:s-ppref-psuf}
Let $S(w) = (S_1(w), \ldots, S_n(w))$ be the sequence of sets where for each $1 \leq i \leq n$, 
$S_i(w)$ is the set of symbols $c \in \Sigma$ such that 
$\phi(c)$ has an occurrence starting at position $i$ in $w$.
Let $\ppref(w)$ be the set of positions $i$ in  $w$ such that 
prefix $w[1 \ldots i]$ is a proper suffix of some image in $\images{\phi}$.
Let $\psuf(w)$ be the set of positions $i$ in  $w$ such that 
suffix $w[i \ldots n]$ is a proper prefix of some image in $\images{\phi}$.
Then, $S(w)$, $\ppref(w)$, and $\psuf(w)$ can be computed in $O(m + n + \mathit{occ})$ expected time.
\end{lemma}

\begin{proof}
After constructing the Aho--Corasick automaton for the set $\images{\phi}$ in $O(m)$ expected time,
using perfect hashing~\cite{conf/focs/1982/fredman} for trie node representation,
we process $w$ through the automaton to compute $S(w)$, which takes $O(n + \mathit{occ})$ time.

To compute $\psuf(w)$,
let $q$ be the node reached after processing the entire $w$.
Then, $\str{q}$ is the longest suffix of $w$ that is also a prefix of some $\phi(c)$ for $c \in \Sigma$.
Let $i$ be the position such that $\str{q} = w[i \ldots n]$.
We have $i \in \psuf(w)$ if and only if $q$ is not a leaf, 
that is, when $\str{q}$ is a proper prefix of $\phi(c)$.
The remaining positions in $\psuf(w)$ can be obtained  
by iteratively following the failure link: 
at each step, we set $q$ to be the node pointed to by the failure link,
check whether $q$ is a leaf, and stop once $q$ becomes the root.
Thus, computing $\psuf(w)$ takes $O(\ell) \subseteq O(m)$ time where
$\ell = \max_{c \in \Sigma} |\phi(c)|$. 

Analogously, $\ppref(w)$ can be computed using the Aho--Corasick automaton for the set $\{ \phi(c)^R : c \in \Sigma \}$
and process $w^R$ through it.
\end{proof}

\begin{algorithm}[t]
\caption{An algorithm for deciding whether $\phi$ is \IF on $\{u\}$.}
\label{algo:if-check}
\Input{an injective morphism $\phi: \Sigma^* \to \Gamma^*$, a word $u \in \Sigma^*$;}
\Output{\texttt{true} if and only if $\phi$ is \IF on $\{u\}$;}
\SetKwFunction{Fproc}{factorizable}
\SetKwProg{Fn}{function}{:}{}
\Fn{\Fproc{$\phi, w$}}{
    compute $S(w)$, $\ppref(w)$, and $\psuf(w)$; \;
    $C[1 \ldots n] \gets 0$;\;
    \lFor{$i \in \ppref(w)$}{$C[i] \gets 1$;
    }
    \For{$i \gets 2 \ldots n$}{
        \If{$C[i-1] = 1$}{
            \lFor{$c \in S_i(w)$}{$C[i + |\phi(c)|] \gets 1$;
            }
        }
    }
    \For{$j \in \psuf(w)$}{
        \lIf{$C[j-1] = 1$}{\Return \texttt{true};
        }
    }
    \Return \texttt{false};
}
$w \gets \phi(u)$;
$\phi^R \gets $ the morphism defined by $\phi^R(c) = \phi(c)^R$ for each $c \in \Sigma$;  \;
\Return  \Not  (\Fproc{$\phi, w$} or  \Fproc{$\phi^R, w^R$});\;
\end{algorithm}

\begin{lemma}\label{lem:factorizable}
In $O(m + n + \mathit{occ})$ expected time, 
function \textup{\texttt{factorizable}} in \cref{algo:if-check}
returns \textup{\texttt{true}} if and only if word $w$ 
admits a factorization  $x \cdot y \cdot z$ such that
$x$ is a (possibly empty) proper suffix of an image, $y$ is a  concatenation of images, and 
$z$ is a non-empty proper prefix of an image.
\end{lemma}

\begin{proof}
Let $C$ be a bit array of length $n$ such that $C[i] = 1$ if and only if
$w[1 \ldots i]$ admits a factorization  $x \cdot y$ such that
$x$ is a (possibly empty) proper suffix of an image and $y$ is a  concatenation of images.
The correctness of \textup{\texttt{factorizable}} follows from this definition.

For the time complexity,
after computing $S(w)$, $\ppref(w)$, and $\psuf(w)$  by \cref{lem:s-ppref-psuf} in $O(m + n + \mathit{occ})$ expected time,
the remaining steps also run in 
$O(m + n + \mathit{occ})$ time,
noting that 
$\sum_{i=1}^n |S_i(w)| \in O(\mathit{occ})$.
\end{proof}

\begin{proof}[Proof of \cref{thm:if-algo}]
The correctness and time complexity of \cref{algo:if-check} follow from \cref{lem:factorizable},
with the additional observation that the definition \IF requires at least one of $x$ and $z$ to be non-empty,
hence the two calls to \texttt{factorizable}.
\end{proof}

Note that in the description of the algorithm,
we focus only on the interfered image factorization case in \cref{def:if-morph},
and observe that the inner image factor case can be verified trivially within the stated time complexity.
We also remark that the use of dynamic programming to decide whether a word admits a factorization satisfying a given property
has appeared previously in~\cite{conf/spire/2015/dumitran, conf/cpm/2016/matsuoka}.

\end{document}